\documentclass[11pt,A4paper]{article}

\usepackage[left=30mm,right=30mm,top=25mm,bottom=30mm]{geometry}
\usepackage{labelfig}
\usepackage{epsfig}
\usepackage{epstopdf}
\usepackage{relsize}
\usepackage{float}

\usepackage{xfrac}

\usepackage[percent]{overpic}

\usepackage{relsize}
\usepackage{color}
\usepackage{amsthm,amsmath,amssymb}
\usepackage{booktabs}
\usepackage{mathpazo}
\usepackage{microtype}
\usepackage[T1]{fontenc}
\usepackage{bm}
\usepackage{sectsty}
\usepackage[
	pdftitle={PDFTitle},
	pdfauthor={Hugo Parlier},
	ocgcolorlinks,
	linkcolor=linkred,
	citecolor=linkred,
	urlcolor=linkblue]
{hyperref}

\usepackage{mathtools}

\definecolor{linkred}{RGB}{255,128,128}

\definecolor{linkblue}{RGB}{100, 210, 210}

\usepackage[hang,flushmargin]{footmisc}
\usepackage{enumitem}
\usepackage{titlesec}
	\titlespacing{\section}{0pt}{12pt}{0pt}
	\titlespacing{\subsection}{0pt}{6pt}{0pt}

\titlelabel{\thetitle.\quad}

\makeatletter 

\long\def\@footnotetext#1{%
\H@@footnotetext{%
\ifHy@nesting 
\hyper@@anchor{\@currentHref}{#1}%
\else 
\Hy@raisedlink{\hyper@@anchor{\@currentHref}{\relax}}#1%
\fi 
}}

\def\@footnotemark{%
\leavevmode 
\ifhmode\edef\@x@sf{\the\spacefactor}\nobreak\fi 
\H@refstepcounter{Hfootnote}%
\hyper@makecurrent{Hfootnote}%
\hyper@linkstart{link}{\@currentHref}%
\@makefnmark 
\hyper@linkend 
\ifhmode\spacefactor\@x@sf\fi 
\relax 
}%

\ifFN@multiplefootnote%
\renewcommand*\@footnotemark{%
\leavevmode 
\ifhmode 
\edef\@x@sf{\the\spacefactor}%
\FN@mf@check 
\nobreak 
\fi 
\H@refstepcounter{Hfootnote}%
\hyper@makecurrent{Hfootnote}%
\hyper@linkstart{link}{\@currentHref}%
\@makefnmark 
\hyper@linkend 
\ifFN@pp@towrite 
\FN@pp@writetemp 
\FN@pp@towritefalse 
\fi 
\FN@mf@prepare 
\ifhmode\spacefactor\@x@sf\fi 
\relax%
}%
\fi 

\makeatother 

\theoremstyle{plain}
\newtheorem{theorem}{Theorem}[section]
\newtheorem{proposition}[theorem]{Proposition}
\newtheorem{lemma}[theorem]{Lemma}

\newtheorem{property}[theorem]{Property}

\makeatletter
\newtheorem*{rep@theorem}{\rep@title}
\newcommand{\newreptheorem}[2]{%
\newenvironment{rep#1}[1]{%
 \def\rep@title{#2 \ref{##1}}%
 \begin{rep@theorem}}%
 {\end{rep@theorem}}}
\makeatother

\newreptheorem{theorem}{Theorem}
\newreptheorem{corollary}{Corollary}

\theoremstyle{definition}
\newtheorem{definition}[theorem]{Definition}

\newtheorem{remark}[theorem]{Remark}

\newcommand{\R}{{\mathbb R}}

\newcommand{\Hyp}{{\mathbb H}}

\newcommand{\Z}{{\mathbb Z}}

\newcommand{\A}{{\mathcal A}}
\newcommand{\CC}{{\mathcal C}}

\newcommand{\OO}{{\mathcal O}}

\newcommand{\g}{{\mathrm{geod}}}
\newcommand{\G}{{\mathcal G}}

\newcommand{\M}{{\mathcal M}}
\newcommand{\PSL}{{\rm PSL}}

\newcommand{\arcsinh}{{\,\rm arcsinh}}

\newcommand{\Sum}{\mathlarger{\mathlarger{ \sum}}}

\newcommand{\teta}{\eta_\circ}
\newcommand{\trun}{\circ}
\newcommand{\tteta}{\eta_{\mathrel{\scalebox{0.4}{$\odot$}}}}
\newcommand{\dt}{{\mathrel{\scalebox{0.4}{$\odot$}}}}
\newcommand{\OB}{\OO\hspace{-0.1em} \beta}
\newcommand{\OBM}{{\OO\hspace{-0.1em} \beta}_{\hspace{-0.1em}M}}
\newcommand{\OM}{\OO_{\hspace{-0.1em}M}}
\newcommand{\OMA}{\OO_{\hspace{-0.1em}M}^{\,\alpha}}
\newcommand{\OBA}{{\OO\hspace{-0.1em} \beta}^{\,\alpha}_{\hspace{-0.1em} M}}

\newcommand{\curves}{\CC(\Sigma)}
\newcommand{\arcs}{\A(\Sigma)}
\newcommand{\mcg}{{\rm Mod}}

\sectionfont{\large \bfseries}
\subsectionfont{\normalsize}

\long\def\symbolfootnote[#1]#2{\begingroup%
\def\thefootnote{\fnsymbol{footnote}}\footnote[#1]{#2}\endgroup}

\def\blfootnote{\xdef\@thefnmark{}\@footnotetext}

\begin{document}

{\Large \bfseries 
Geodesic and orthogeodesic identities on hyperbolic surfaces}

{\large Hugo Parlier\symbolfootnote[1]{\normalsize Supported by the Luxembourg National Research Fund OPEN grant O19/13865598.\\
{\em 2020 Mathematics Subject Classification:}\\Primary: 32G15, 37D40, 57K20. Secondary: 30F60, 37E35, 53C22.\\
{\em Key words and phrases:} Identities, orthogeodesics, moduli spaces.}}

{\bf Abstract.} 
The lengths of geodesics on hyperbolic surfaces satisfy intriguing equations, known as identities, relating these lengths to geometric quantities of the surface. This paper is about a large family of identities that relate lengths of closed geodesics and orthogeodesics to boundary lengths or number of cusps. These include, as particular cases, identities due to Basmajian, to McShane and to Mirzakhani and Tan-Wong-Zhang. In stark contrast to previous identities, the identities presented here include the lengths taken among {\it all} closed geodesics. 

\vspace{1.2cm}

\section{Introduction}

The relationship between lengths of closed geodesics, hyperbolic surfaces, and their underlying moduli spaces is intriguing and has been studied from many different perspectives. One exciting feature appeared in the work of Basmajian \cite{Basmajian} and McShane \cite{McShane} who discovered two different {\it identities} which relate infinite sums, where the terms depend on lengths of geodesics, to geometric quantities. In the case of Basmajian, the geometric quantity is the boundary length of the surface and in the original McShane identity, it is the length of a horocycle surrounding a cusp. Both of these identities have been generalized in multiple contexts, most famously perhaps being the generalization of the McShane identity by Mirzakhani who then showed how to use the identity to compute volumes of moduli spaces \cite{Mirzakhani}. 

The exact statements of these original identities will be given below (and indeed scrutinized carefully), and although they seem at first glance to be unrelated, they share an aesthetic feature of simplicity. The Basmajian identity is a sum over all orthogeodesics of terms $\log(\coth(\ell/2))$ that depend naturally on the length $\ell$ of the orthogeodesic. The sum gives the boundary length. The original McShane identity is on a one cusped torus, and is a sum over all simple closed geodesics of terms $\frac{1}{e^\ell+1}$ depending only on the length $\ell$ of the simple closed geodesic. The sum is equal to $1/2$. Their many successors have plenty of redeeming features, but are no match in terms of simplicity of the statement. One of the goals of the present paper is to show how these identities are related, and are in fact part of a family of identities that involve both lengths of closed geodesics and orthogeodesics.

The identities will be for orientable hyperbolic surfaces (of finite type) with non-empty boundary. Points in the moduli space $\M(\Sigma)$ of hyperbolic surfaces $X$ homeomorphic to fixed given topological surface $\Sigma$ with be denoted $\M(X)$. A hyperbolic surface $X$ will be required to have either simple closed geodesics or cusps as boundary. Non-empty boundary means that $\pi_1(\Sigma)$ is a free group, and conjugacy classes of elements in $\pi_1(\Sigma)$ correspond to (free) homotopy classes of closed curves $\curves$, which it turn corresponds to $\G(X)$, the set of oriented closed geodesics of $X$. $\G(X)$ will be the set of oriented closed geodesics of $X$, and can thus be viewed from an algebraic, topological or geometric viewpoint. In particular, if you move $X$ in its moduli space, then $\G(X)$ remains the same from the algebraic or topological viewpoint. Similarly, we use $\OO=\OO(X)$ for the set of all oriented orthogeodesics of $X$: these are geodesic segments with endpoints on and orthogonal to $\partial X$. If one of the endpoints is a cusp $c$, an orthogeodesic that leaves or returns to or from $c$ is of infinite length so the term orthogonal refers to it being orthogonal to a (or any) horocyclic boundary of $c$.

The identities are all given by a choice of subset $M \subset \curves$ which is {\it coherent}, by which it is meant that $M$ must satisfy a property which we now outline. An element in $M$ is a closed curve, and if it is not simple and primitive, it might contain proper subloops. Being coherent means that if $\gamma\in M$, then none of the subloops of $\gamma$ belong to $M$ (see Section \ref{ss:curves} for a more detailed definition). Note that orthogeodesics can also contain subloops, and if an orthogeodesic has a subloop freely homotopic to an element $\gamma\in \curves$, we say that it {\it supports} $\gamma$.

Some immediate examples of coherent markings are when $M$ is empty or $M$ is any single curve. Any coherent marking (finite or infinite) can be obtained as follows: take any curve $\gamma$ and add it to $M$, now add any other curve that is not homotopic to a subloop of a curve in $M$, and so forth. Note that there is no requirement that if $\gamma \in M$, then $\gamma^{-1}$ also lie in $M$. A perhaps less obvious example: $M$ is the set of all simple primitive closed geodesics (with both orientations). Or, alternatively, choose exactly one orientation for each simple closed geodesic, and let $M$ be the set of those. More generally, any subset of a coherent marking is coherent, and using this fact it is not hard to convince oneself that for fixed $\Sigma$, there are uncountably many coherent markings. Some of them are mapping class group invariant, but in general they are not. Other examples of a coherent markings can be obtained by fixing a filling closed curve $\gamma$, and taking the full mapping class group orbit of $\gamma$ (thus $M = \mcg_\Sigma(\gamma)$). 

This allows, given a coherent marking $M$, for a separation of $\OO$ into two sets: those that are supportive of elements of $M$, and those that are not. Note that, by definition, all simple orthogeodesics are unsupportive. Among the unsupportive elements of $\OO$ (denoted $\OM$), certain are {\it peripheral}, meaning that together with a boundary arc they form a homotopy class in $M$. A more detailed definition is given in Section \ref{sec:prelim}, but for now we denote the set of orthogeodesics peripheral to $M$ by $\OM$ and write $\partial \eta \ni \alpha$ if $\eta$ is peripheral to $\alpha$.

The abstract identity is as follows:
\begin{theorem}\label{thm:abstract}
Let $X$ be a hyperbolic surface with $\partial X \neq \emptyset$ consisting of cusps and/or simple closed geodesics. Let $\beta$ be a boundary element, $M$ a coherent marking of $\curves$ and let $\OB$ be the set of orthogeodesics leaving from $\beta$.

If $\beta$ is a closed geodesic, then:
\begin{equation*}
\ell_X(\beta) = \Sum_{\eta \in \OBM} \phi(\eta) + \Sum_{\alpha \in M}\left(\Sum_{\eta \in \OBA} \psi(\eta)\right)
\end{equation*}
where $\phi(\eta)$ is a measure that depends only on the length of $\eta$, and $\psi(\eta)$ is a measure depending only the length of $\eta$ and $\alpha$ (the curve $\eta$ is peripheral to). 

When $\beta$ is a cusp:
$$
2 = \Sum_{\eta \in \OBM} \phi_\trun(\eta) +\Sum_{\alpha \in M}\left(\Sum_{\eta \in \OBA} \psi_\dt(\eta)\right)
$$
where $\phi_\trun(\eta)$ is a measure that depends only on the truncated length of $\eta$, and $\psi_\dt(\eta)$ depends only the length of doubly truncated length of $\eta$ and $\alpha$ (the curve $\eta$ is peripheral to). \end{theorem}

The truncated length of an orthogeodesic leaving from a cusp is the restriction of the orthogeodesic to the subarc that leaves from the boundary of a standard horocyclic neighborhood of the cusp. For an orthogeodesic that leaves and returns to the same cusp, the doubly truncated length is the length of the subarc that leaves from the cusp neighborhood until it returns to the cusp neighborhood for the last time. (Note that truncated orthogeodesics might return multiple times to the the horocyclic neighborhood.) The functions (or measures) are often called gap functions or just gaps because they are measures of segments of the boundary associated to index terms. They are not only abstract, and can be quantified. Here is a first quantification.

\begin{theorem}\label{thm:boundarylength}
Let $X\in \M(\Sigma)$ be a hyperbolic surface of finite type with $\partial X \neq \emptyset$ containing at least one simple closed geodesic, and $M$ a coherent marking of $\curves$. The lengths of curves and orthogeodesics of $X$ satisfy
\begin{eqnarray*}
\ell(\partial X) = &
\Sum_{\eta \in \OM} & \log \left( \coth^2 (\eta/2) \right)\\
 +& \Sum_{\alpha \in M} \,\, \Sum_{\eta \in \OMA} &\log\left( \frac{\cosh(\alpha/2) + \sqrt{\cosh^2(\alpha/2) + \cosh^2(\eta/2)-1}}{\sinh(\alpha/2) + \sqrt{\cosh^2(\alpha/2) + \cosh^2(\eta/2)-1}} \right).
\end{eqnarray*}
\end{theorem}

Note that any $\eta \in \OM$ appears twice: once in the first sum and once in the second. It would of course be possible to regroup these terms, but geometrically these terms describe different phenomena.

The above equation can be expressed in terms of half-traces (that is the hyperbolic cosine of half-lengths: $a=\cosh(\alpha/2)$, $b=\cosh(\beta/2)$, and $t=\cosh(\eta/2)$). By taking the exponential, the following product formula appears:
\begin{equation}\label{eq:halftraceformulation}
b + \sqrt{b^2-1} =\left( \prod_{\eta} \frac{t^2}{t^2-1}\right) \prod_{\alpha \in M} \left( \prod_{\partial \eta \ni \alpha} \left( \frac{a + \sqrt{a^2+ t^2-1 }}{\sqrt{a^2-1} + \sqrt{a^2+ t^2-1 }}\right) \right) 
\end{equation}
where the first product is taken over all orthogeodesics unsupportive of $M$, and the last product is over all orthogeodesics unsupportive of $M$ {\it and} peripheral to $\alpha$. 

The first sum (or product in the half-trace formulation) is exactly the same term as the term in the Basmajian identity \cite{Basmajian}, and is the projection of a copy of $\beta$ onto itself (following the reverse direction of the orthogeodesic). And in fact, when $M = \emptyset$, the above identity is the Basmajian identity: all orthogeodesics are unsupportive, the first sum is empty, and the second sum sums over all orthogeodesics leaving from $\beta$.

When $M$ however is the set of all (primitive) simple closed geodesics, taken with both possible orientations and including boundary curves, the above identity is an expression of the McShane identity for surfaces with geodesic boundary due to Mirzakhani \cite{Mirzakhani} and Tan-Wong-Zhang \cite{Tan-Wong-Zhang}. This is another extremal case of the identity, because it is a sum over simple orthogeodesics. And simple orthogeodesics are always unsupportive. Hence, the first sum has the "smallest" possible index set, in contrast with the Basmajian case where the first sum has the "largest" possible set. The expression of the identity is clearly different from previous formulations however, and this is due to several factors. The gap functions in the work of Mirzakhani and Tan-Wong-Zhang are all computed using the boundary curves of the associated pair of pants. In particular, the length of $\beta$ plays a part in the function and the functions depend on two factors. Here the gaps are divided into parts: one part corresponding the Basmajian type term, which only depends on the associated ortholength, and two different gaps on each side of the orthogeodesic, which have another dynamical interpretation to be be discussed later, and which depend on two geometric quantities, the ortholength and the length of $\alpha$. This difference in interpretation of the gaps shows how to express the length of $\beta$ via a sum of gaps which depend on geometric quantities which do not involve $\beta$, similarly to the original Basmajian identity.

By taking an appropriate geometric limit, the identity can be quantified as follows on a surface with cusps. Note that one of the sum disappears. 

\begin{theorem}\label{thm:ncusps}
Let $X$ be a hyperbolic surface with $\partial X \neq \emptyset$ consisting of $n>0$ cusps, and $M$ a coherent marking.
\begin{equation*}
n = \frac{1}{2}\Sum_{\alpha \in M} e^{-\frac{\alpha}{2}} \left( \Sum_{\eta \in \OMA} e^{-\frac{\tteta}{2}}\right)
\end{equation*}
\end{theorem}

When $M$ is the set of all primitive simple closed geodesics, the above expression is a reformulation of the original McShane identity \cite{McShane}. This will be explained in Section \ref{ss:relationships}.

Both quantifications of the identity (Theorems \ref{thm:boundarylength} and \ref{thm:ncusps}) are obtained as sums of a quantification of the abstract identity, which takes into account one boundary component at a time:

\begin{theorem}\label{thm:oneatatime}
If $\beta$ is a boundary simple closed geodesic of $X$, it satisfies:
\begin{eqnarray*}
\ell(\beta) = &
\Sum_{\eta \in \OBM} & \log \left( \coth^2 (\eta/2) \right)\\
 +& \Sum_{\alpha \in M} \,\, \Sum_{\eta \in \OBA}  &\log\left( \frac{\cosh(\alpha/2) + \sqrt{\cosh^2(\alpha/2) + \cosh^2(\eta/2)-1}}{\sinh(\alpha/2) + \sqrt{\cosh^2(\alpha/2) + \cosh^2(\eta/2)-1}} \right).
\end{eqnarray*}

If $\beta$ is a cusp:
\begin{equation*}
2 = 2 \Sum_{\eta \in \OBM} e^{-\teta} +  \Sum_{\alpha \in M} \,\, \left(\Sum_{\eta \in \OBA} e^{-\frac{\alpha+\tteta}{2}}\right)
\end{equation*}
where $\teta$ is the truncated length of $\eta$ and, for orthogeodesics that leave and return to $\beta$, $\tteta$ is the doubly truncated length. 
\end{theorem}

{\bf Organization.}

The article is organized as follows. Section \ref{sec:prelim} contains a lot of the groundwork: after having introduced a bit of notation, coherent markings are properly introduced and studied, followed by a detailed computation of the gaps. In Section \ref{sec:proof}, the abstract identity is proved. In the final section, the quantified identities are formulated by inserting the previously computed gaps, and the relationship to existing identities is then discussed.

{\bf Acknowledgments.}

These results are the fruit of a fascination for the beautiful geometric identities mentioned above and enlightening discussions about them. A particular thanks to Greg McShane, Martin Bridgeman and Maryam Mirzakhani. Many ideas and a lot of inspiration also come from my recent collaboration with Ara Basmajian and Ser Peow Tan: thanks for keeping us sane and for making working on identities so much fun. Thank you Binbin Xu for comments on a first draft and Francis Lazarus for a discussion that made its way into Proposition \ref{prop:looproots}.

\section{Background, notation and preliminaries}\label{sec:prelim}

Let $\Sigma$ be a topological orientable surface with a non-empty set of marked points and negative Euler characteristic. Take $X$ to be a hyperbolic surface, homeomorphic to $\Sigma$ and with either cusp or geodesic boundary. The moduli space $\M(\Sigma)$ is the set of such $X$ up to isometry and $\mcg(\Sigma)$ is the (full) mapping class group of $\Sigma$, that is homeomorphisms up to isotopy. 

\subsection{Curves, arcs and geodesics}\label{ss:curves}

In a bit of somewhat non-standard notation, $\CC(\Sigma)$ will designate the set of {\it all} curves on $\Sigma$, by which we mean the set of free homotopy classes of oriented closed curves, including curves peripheral to boundary. Elements of $\CC(\Sigma)$ are not parametrized, but are oriented, and include all non-primitive elements. In particular, there are infinitely many homotopy classes peripheral to boundary. Because of our condition on $X$, via the unicity of geodesics in free homotopy classes, elements of $\CC(\Sigma)$ correspond to the set of all closed geodesics $\G(X)$, including those that are not primitive and boundary elements. There is a slight matter of convention here: if $X$ has cusp boundary, it might not make much sense to consider a cusp as a boundary geodesic, much less an oriented boundary geodesic or a power of one, but this slight discrepancy between $\CC(\Sigma)$ and $\G(X)$ as sets of free homotopy classes is not a fundamental issue and will be discussed further below. 

Similarly, $\A(\Sigma)$ will be the set of topological arcs on $\Sigma$ with endpoints on the marked points and up to homotopy preserving the endpoints. Note there is no notion of primitive for elements of $\A(\Sigma)$ (said otherwise, all elements of $\A(\Sigma)$ are primitive). Given $X$, the geometric realization of $\A(\Sigma)$ is the set $\OO(X)$ consisting of orthogeodesics of $X$: these are geodesic arcs with endpoints orthogonal to the boundary of $X$. Again, if $X$ has cusps, the term orthogeodesic might seem inappropriate for the infinite length geodesics travelling to a cusp, but in fact these geodesics are, in the a region sufficiently deep in the cusp, orthogonal to the boundary of any horocylic neighborhood of the cusp.

Before introducing markings, we begin by making some observations about $\CC(\Sigma)$ and $\A(\Sigma)$.

{\it Closed curves, arcs and their subloops}

Let $\gamma \in \CC(\Sigma)$ which we always think of being in minimal position, meaning that minimizes self-intersection number. For the most part, we will think of closed curves as being realized topologically in minimal position, and to avoid further confusion, in a way such that all intersection points are double. Hence in an intersection point, because curves are oriented, there are exactly two outgoing directions, each associated to an incoming direction. If $\gamma$ is primitive, meaning not the power of another curve, then its realization as a closed geodesic in $\G(X)$ (for any $X$) is in minimal position. There is a subtlety here when $\gamma$ is not primitive however: its unique closed geodesic representative will be some number of iterates of the associated primitive closed geodesic. While this geometric representative minimizes {\it transversal} self-intersection, of course there are infinitely many double points, but of course there are different ways of getting around this difficulty. In any event, we denote by $i(\alpha,\alpha)$ the minimal number of double points among topological representations (higher order multiple points are counted with multiplicities so that minimality is reached by certain representatives with double points). 

Associated to a curve in $\CC(\Sigma)$ are (oriented and proper) subloops. Given any curve with self-intersection points, in each intersection point (which we suppose double) there are exactly two (not necessarily) distinct subloops, obtained by taking one of the two outgoing paths until it returns. In any point that is not a self-intersection point, there is a unique loop beginning and ending given by the curve itself, and so there are no proper subloops in that point. Hence, any closed curve $\alpha$, has at most $2 i (\alpha, \alpha)$ proper subloops, each corresponding to (not necessarily distinct) free homotopy classes of oriented closed curves. Further observe that this set of elements of $\CC(\Sigma)$ associated to a given $\alpha$ does {\it not} depend on the choice of representation of $\alpha$. Indeed, given an intersection point of a curve, following the associated loops as the point moves under a free homotopy, it is easy to see that the two free homotopy classes do not change. Similarly, arcs can have proper subloops, corresponding to homotopy classes of closed curves. This brings us to the following:

\begin{definition} Let $\alpha$ and $\beta$ be curves in $\CC(\Sigma)$. We say that $\beta$ is supported by $\alpha$ if $\beta$ is freely homotopic to a proper subloop of $\alpha$ and we write $\beta \xhookrightarrow{} \alpha$.

Similarly, for $\eta \in \A(\Sigma)$, we say $\beta$ is supported by $\eta$ if it is freely homotopic to a proper subloop of $\eta$ and we write $\beta \xhookrightarrow{} \eta$. 
\end{definition}

We now make a few observations about curves, arcs and their proper subloops. 

\begin{proposition}\label{prop:looproots}
Let $\beta \xhookrightarrow{} \alpha$ and suppose $\beta$ is not prime, meaning $\beta = (\beta')^k$ for some natural number $k>1$. Then $\beta' \xhookrightarrow{} \alpha$. More generally $(\beta')^l \xhookrightarrow{} \alpha$ for all $l\leq k$. 
\end{proposition}

\begin{proof}
As will be detailed in the sequel, for a given hyperbolic structure this is certainly true for any closed geodesic representative. This comes from the fact that a geodesic only loops $k$ times around a given geodesic $\beta'$ if it belongs to a small enough collar around $\beta'$. And the width of the collar depends only on the length of $\beta'$, and is decreasing. In particular, if you enter the collar for $\beta$, you've entered the collar for each $(\beta')^k$. See Remark \ref{rem:collar} for more details about the collar. 

Although we make no use of this in the sequel, it is interesting to point out that this is really a topological fact, and does not even depend on having $\beta$ in minimal position. To see this, the easiest is probably to look at $\beta$ in the annular cover corresponding to $\beta'$ (this is a finite cover where $\beta'$ has a lift freely homotopic to a simple primitive loop which generates the annulus). The curve $\beta$ has a lift $\tilde{\beta}$ that is freely homotopic to multiple of generator of the annulus. Now it is not too difficult to convince oneself that $\tilde{\beta}$ contains a subloop that is homotopic to the generator, which doesn't contain the basepoint (this last point is key, otherwise it won't project to the desired loop). 
\end{proof}

Also observe that $\alpha \in \CC(\Sigma)$ is simple (and primitive) if and only if it does not support {\it any} curves, and similarly for $\eta \in \A(\Sigma)$. However, if $\alpha$ is a power of a simple curve, then it supports all lesser powers. More generally, any curve that is a proper power supports all lesser powers.

More generally, one might hope that its a transitive property (if $\beta \xhookrightarrow{} \alpha$ and $\gamma \xhookrightarrow{} \beta$, then $\gamma \xhookrightarrow{} \alpha$) but this is not always the case, see Figure \ref{fig:nottransitive}.

\begin{figure}[h]
\leavevmode \SetLabels
\L(.31*.36) $\xhookrightarrow{}$\\%
\L(.65*.36) $\xhookrightarrow{}$\\%
\endSetLabels
\begin{center}
\AffixLabels{\centerline{\includegraphics[width=16cm]{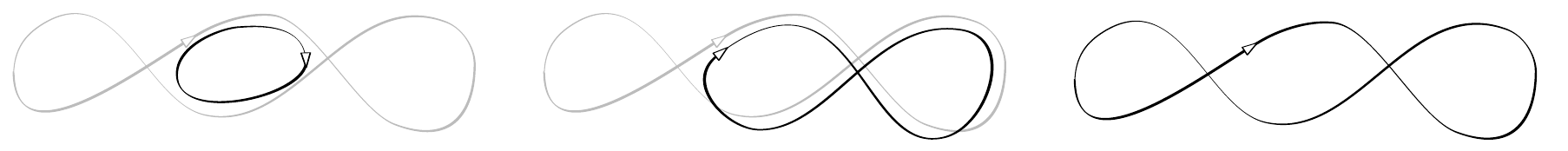}}}
\vspace{-24pt}
\end{center}
\caption{The left most curve is not a subloop of the right most curve}
\label{fig:nottransitive}
\end{figure}

We now pass to the geometric counterparts of these notions. 

{\it Closed geodesics and orthogeodesics}
 
Now consider $X\in \M(\Sigma)$. To any element of $\curves$ that is not peripheral to boundary, there is a unique geodesic representative. Likewise, any interior element of $\G(X)$ corresponds to a unique element of $\curves$. This is also true for boundary geodesics and their corresponding peripheral homotopy classes, but not quite to cusps. The cusp is the geometric realization of all homotopy classes peripheral to it. There is however a complete one to one correspondence between elements of $\arcs$ and those of $\OO(X)$. Via this correspondence, closed geodesics and orthogeodesics of $X$ inherit the notion of {\it support}.

{\bf Convention.} Throughout the article, geometric quantities, namely lengths realized on an underlying hyperbolic surface $X$, will be denoted by the same symbol as the corresponding homotopy classes in $\curves$ and $\arcs$. For example, with this convention the quantity $\cosh(\alpha/2)$ for $\alpha \in \curves$, can be thought of as a function of $X$ in $\M(\Sigma)$ which associates to $X$ the hyperbolic cosine of half of the length of the unique geodesic in the free homotopy class of $\alpha$.

Oriented geodesics on $X$ naturally correspond to subset of vectors in $T^{1}(X)$, the unit tangent bundle of $X$. To each $\gamma \in \G(X)$, we can associate a so-called stable neighborhood \cite{Basmajian2}, but here it is more relevant to associate another subset of $T^{1}(X)$ to $\gamma$.

To define it, introduce a bit of notation: given $v \in T^1(X)$, we let $\g_v$ be the complete oriented geodesic on $X$ obtained by exponentiating $v$ forwards and backwards, but only taking the forwards orientation (given by $v$). 

\begin{definition}
Let $\gamma \in \curves$. The loop set of $\gamma$, denoted $L_X(\gamma)$, is the subset of $T^1(X)$ consisting of all vectors belonging to oriented geodesics that form a loop freely homotopic to $\gamma$. Stated differently:
$$
L_X(\gamma)= \{ v\in T^1(X) \mid  \g_v \mbox{ contains a geodesic loop freely homotopic to $\gamma$}\}.
$$
\end{definition}

This is well-defined, even if $\gamma$ is a homotopy class peripheral to boundary, and even if this boundary is geometrically realized as a cusp. 

\begin{remark}\label{rem:collar}To $\gamma$ we can associated a (generally immersed) collar (or stable neighborhood) which is related to the loop set of $\gamma$. If you look at an annular lift of $\gamma$ (a finite lift where $\gamma$ is simple), by an adaptation of the classical collar lemma, see for instance \cite{BPT}, one obtains an embedded collar of half-width exactly 
$$
w(\gamma) = \arcsinh\left( \frac{1}{\sinh(\gamma/2)}\right)
$$
where to simply notation, $\gamma$ also denotes the length of the geodesic in $\G(X)$ corresponding to $\gamma$. The stable neighborhood is the projection of this collar to $X$. By construction, it includes all points of $X$ distance at most $w(\gamma)$ 
from $\gamma$. Now if a geodesic comes at least $w(\gamma)$ close to $\gamma$, it either crosses $\gamma$ immediately, by which it is meant that the geodesic follows $\gamma$ until it intersects $\gamma$ transversally, or it belongs to $L_X(\gamma)$. 
\end{remark}

\subsection{Coherent markings and peripheral orthogeodesics}

We can now introduce markings which are essential to the index sets of the identities. 

A {\it marking} of a set of curves is a choice of a subset $M$ of $\CC(\Sigma)$, which is allowed to be empty. (Equivalently, one could think of a marking as a map from $\CC(\Sigma)$ to $\Z_2$, where the image is $1$ for all marked curves, and $0$ for the others.)

We say that a marking $M$ is {\it coherent} if it satisfies the property that if $\alpha \in M$, then all curves supported by $\alpha$ are not. In other words, if $\alpha\in M$ and if $\beta \xhookrightarrow{} \alpha$, then $\beta \not\in M$. 

Here are some examples of coherent markings:
\begin{enumerate}
\item $M = \emptyset$. This example will give rise the index set of the Basmajian identity in the sequel. 
\item $M=\{\alpha\}$ where $\alpha\in \CC(\Sigma)$ is any choice of curve. Similarly, you can construct finite sets of curves by taking any finite set, and remove curves one by one that are supported by other curves. Note that there may be more than one way to do this. 

\item\label{ex:simple}$M=\{\alpha \mid \alpha\in \CC(\Sigma)\mbox{ is simple and primitive.}\}$. Note that this includes curves peripheral to boundary, and if $\alpha \in M$, then $\alpha^{-1} \in M$. This example will give rise the index set of McShane type identities as will be explained in Section \ref{sec:exampleidentities}. 

\item\label{ex:orbit}$M = \{ \phi(\alpha) \mid \phi \in \mcg(\Sigma)\}$ for a choice of $\alpha \in \CC(\Sigma)$. These are the curves obtained in the mapping class group orbit of a curve. For certain choices of $\alpha$, it might happen that $\phi(\alpha) = \phi'(\alpha)$ for different $\phi, \phi' \in \mcg(\Sigma)$, but in this case, the curve is not counted with multiplicity. An example where this {\it doesn't} happen is when $\alpha$ is a filling curve (its complementary region is a collection of topological disks with at most one marked point). 
\end{enumerate}

Here are some immediate properties of coherent markings. 

\begin{property} Coherent markings satisfy the following:
\begin{itemize}
\item Any coherent marking can be obtained by taking curves in succession $\{\gamma_1,\gamma_2, \hdots \}$ and verifying that $\gamma_k$ is not supported by $\{\gamma_1,\hdots,\gamma_{k-1}\}$.
\item Any subset $M' \subset M$ of a coherent marking $M$ is a coherent marking.
\item There are uncountably many coherent markings. 

\item If $\alpha\in M$ then for all integers $k>1$, $\alpha^k \not\in M$. 
\end{itemize}
\end{property}

\begin{proof}
The first property follows from the fact that the set of elements of $\CC(\Sigma)$ is countable.

The second property is immediate, because if a subset of $M$ violates coherency, then so does $M$.

For the third property, it suffices to find a marking of infinite cardinality, and then to consider all subsets of this marking. The examples \ref{ex:simple} and \ref{ex:orbit} above provide such markings.

The last property follows from the definition of a coherent marking and that roots of a non-primitive curve are subloops. 
\end{proof}

Given a coherent marking $M$, $\A(\Sigma)$ naturally splits into two sets: those that support at least one element of $M$, and those that don't. An arc $\eta \in \A(\Sigma)$ is said to be {\it unsupportive} if it does not support any element of $M$. 

Certain arcs have the same starting and ending points. For an associated orthogeodesic, these are those that leave and return to the same boundary curve. Naturally associated to an arc $\eta$, whose endpoints are both the same marked point $p$, are two elements of $\curves$. These the two curves obtained by looking at $\eta$ on $\Sigma \setminus \{p\}$ as a closed curve, and pushing it slightly away from $p$ (see Figure \ref{fig:push}). 

\begin{figure}[h]
\leavevmode \SetLabels
\L(.567*.74) $\eta$\\%
\L(.45*.8) $p$\\%
\endSetLabels
\begin{center}
\AffixLabels{\centerline{\includegraphics[width=12cm]{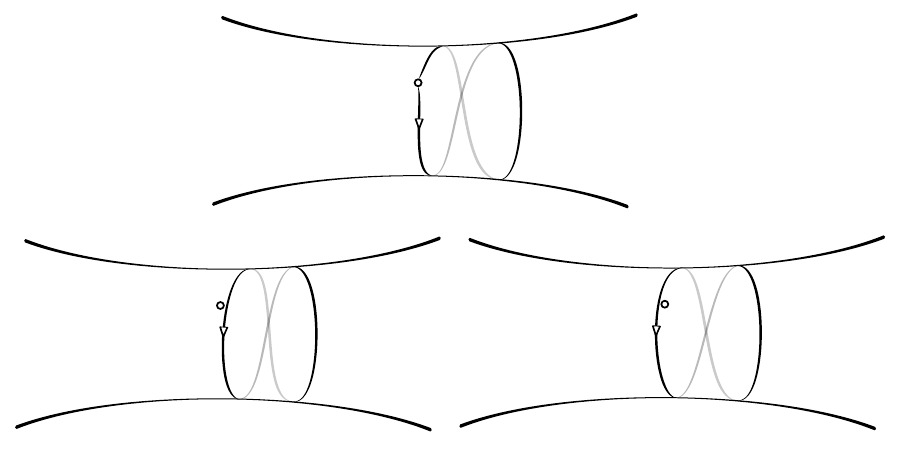}}}
\vspace{-24pt}
\end{center}
\caption{The two curves obtained by pushing $\eta$ slightly away from $p$ }
\label{fig:push}
\end{figure}

Often these two curves are distinct, but there is one case where they are not. Take $\Sigma$ to be a torus with a single marked point, and take $\eta$ to be a simple arc. Then the two (oriented) curves are freely homotopic (see Figure \ref{fig:oneholedtorus}).

\begin{figure}[H]
\leavevmode \SetLabels
\endSetLabels
\begin{center}
\AffixLabels{\centerline{\includegraphics[width=5cm]{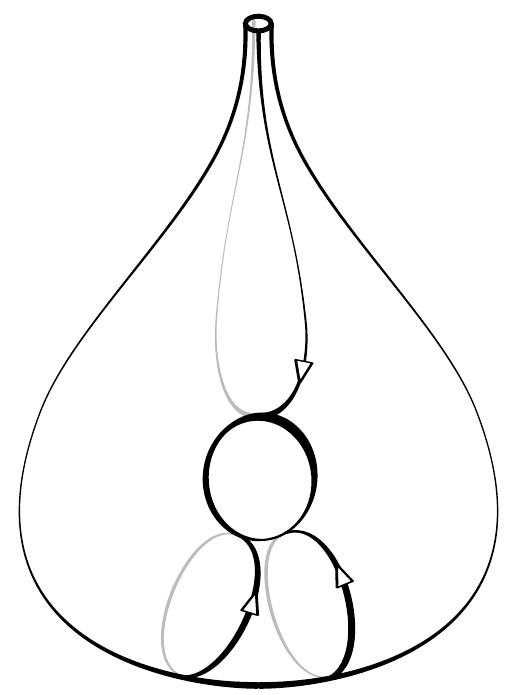}}}
\vspace{-24pt}
\end{center}
\caption{An $\eta$ with identical curves in $\partial \eta$}
\label{fig:oneholedtorus}
\end{figure}

Given an arc $\eta \in \arcs$, we denote the set of these two (not necessarily distinct) curves $\partial \eta$,  and if $\alpha \in \partial \eta$, we say that $\eta$ is {\it peripheral} to $\alpha$.

With this in hand, for a given coherent marking $M$, we see that certain unsupportive arcs are special: we say that an unsupportive arc $\eta$ is {\it peripheral} to $M$ if there exists $\alpha \in M$ such that $\alpha \in \partial \eta$. If an arc $\eta$ is peripheral to $\alpha$, note that via the homotopy on $\Sigma\cup \{p\}$, $\eta$ and $\alpha$ bound a (possibly) immersed cylinder $H$ on $\Sigma$, with $p$ lying on one of its boundaries. Similarly, $p$ and the two curves in $\partial \eta$ form the boundary of an immersed thrice punctured sphere (a pair of pants). The cylinder $H$ we call an immersed {\it half-pants}. See Figure \ref{fig:halfpants} for an example. 

\begin{figure}[H]
\leavevmode \SetLabels
\L(.605*.78) $\eta$\\%
\L(.462*.78) $\eta$\\%
\L(.645*.24) $\alpha$\\%
\endSetLabels
\begin{center}
\AffixLabels{\centerline{\includegraphics[width=4.5cm]{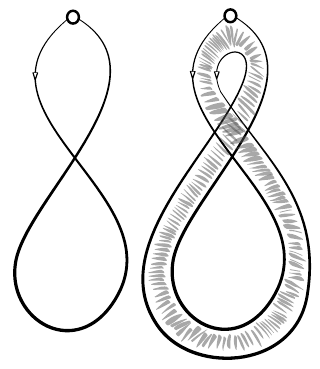}}}
\vspace{-24pt}
\end{center}
\caption{An arc $\eta$ and an associated half-pants}
\label{fig:halfpants}
\end{figure}

\subsection{Gaps and their inclusion}

Associated to an orthogeodesic are two types of gaps. One of them is the gap that appears in the Basmajian identity, and the other is a generalization of the type of gap that appears in McShane type identities.

{\it Type 1: Basmajian type gaps}

For completion we describe, and compute, the gap associated to an orthogeodesic following \cite{Basmajian}. 

Let $\eta$ be an orthogeodesic between boundary elements $\beta$ and $\beta'$. We begin with the case when $\beta$ and $\beta'$ are both simple closed geodesics. If we take a lift of $\eta$ in the universal cover $\Hyp$, it lies between lifts of $\beta$ and $\beta'$, say $\tilde{\beta}$ and $\tilde{\beta}'$. If you take the shortest point projection to $\tilde{\beta}$ of $\tilde{\beta}'$, you obtain a segment $b$ on $\tilde{\beta}$ of length that only depends on $\eta$. The shortest point projection geodesics, together with the segment  $b$ and $\tilde{\beta}'$, form a quadrilateral in $\Hyp$ with two right angles and two ideal points, as portrayed in Figure \ref{fig:basman}. By splitting the quadrilateral in two along the lift of $\eta$, we get a Lambert quadrilateral (called a trirectangle in \cite{BuserBook}), which satisfies
$$
\sinh(b/2) \sinh(\eta) =1.
$$
From this:
$$
b = 2 \arcsinh\left(\frac{1}{\sinh(\eta)}\right) =2 \log \left(\coth\left(\frac{\eta}{2}\right) \right).
$$

\begin{figure}[h]
\leavevmode \SetLabels
\L(.303*.66) $\frac{\phi(\eta)}{2}$\\%
\L(.43*.5) $\eta$\\%
\endSetLabels
\begin{center}
\AffixLabels{\centerline{\includegraphics[width=6cm]{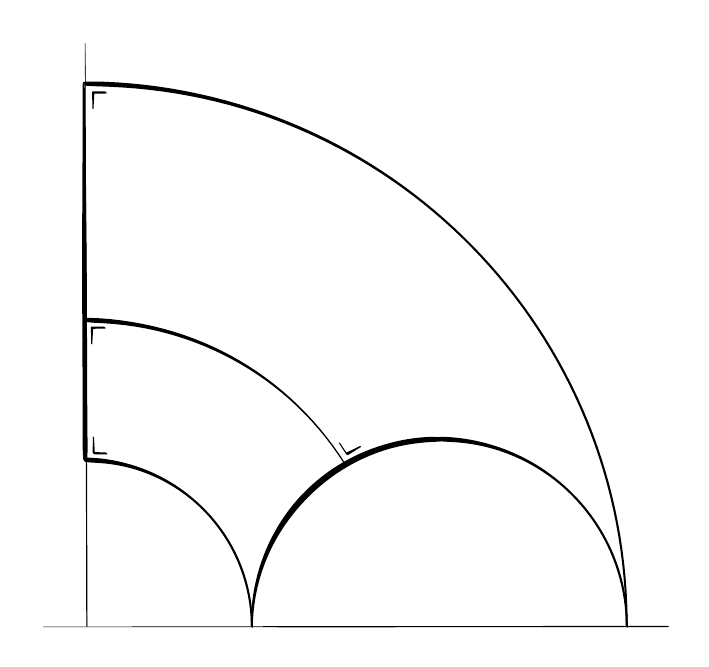}}}
\vspace{-24pt}
\end{center}
\caption{How to compute the Basmajian gap in $\Hyp$}
\label{fig:basman}
\end{figure}

Now $b$ projects to $X$ to a segment along $\beta$ of the same length, say $\phi(\eta)$. 
\begin{remark}
If you double the surface $X$ along its boundary, the two copies of $\eta$ become a closed geodesic of length $2\eta$. On this doubled surface, this closed geodesic has a collar, as observed in Remark \ref{rem:collar}, which is exactly of width $\phi(\eta)/2$. This is not, of course, a complete coincidence.
\end{remark}

If the orthogeodesic goes from a cusp and a closed geodesic, the closed geodesic can be projected onto a horocyclic boundary of the cusp. The gap this time will depend on the truncated length, by which we mean the length of the orthogeodesic that leaves from the cusp neighborhood and arrives on the geodesic. Note that the orthogeodesic is allowed to pass through the cusp neighborhood (as many times as it likes) and that length is still accounted for in the truncated length. In other words, it is only the initial infinite length segment from the cusp to the boundary of the neighborhood that is truncated.

\begin{figure}[h]
\leavevmode \SetLabels
\L(.415*.70) $\frac{\phi(\eta)}{2}$\\%
\L(.474*.47) $\teta$\\%
\endSetLabels
\begin{center}
\AffixLabels{\centerline{\includegraphics[width=6cm]{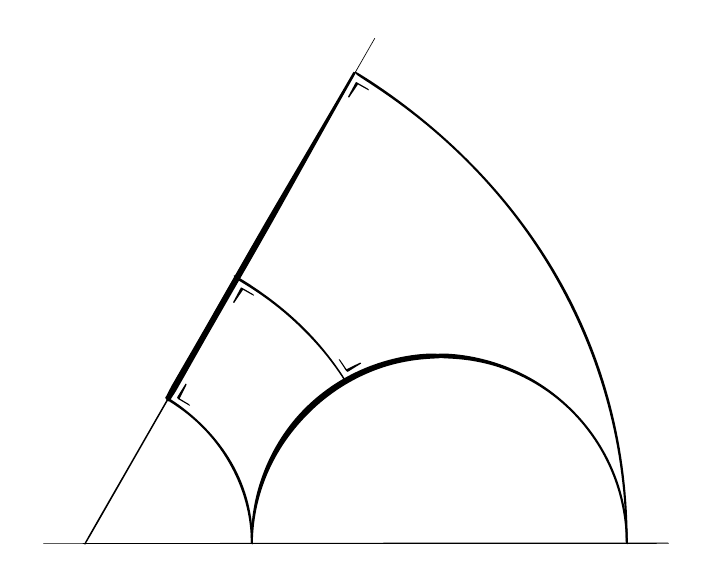}}}
\vspace{-24pt}
\end{center}
\caption{Computing the truncated gap in $\Hyp$}
\label{fig:truncatedbasman}
\end{figure}

If we denote by $\teta$ its truncated length, we obtain a gap of size
$$
\phi(\eta)= \frac{2}{e^{\teta}}
$$
associated to the orthogeodesic $\eta$. One way to compute this is by seeing this gap as a limit of the gap along a collar of $\beta$ as the length of $\beta$ goes to $0$. Alternatively, one can do a direct computation using the geometry portrayed in Figure \ref{fig:truncatedbasman}.

Note that if we have an orthogeodesic ending in a cusp, we do {\it not} consider a type of doubly truncated orthogeodesic and the gap associated to this orthogeodesic is always $0$. As we will see now, for orthogeodesics that return to the same cusp, the second type of gap {\it will} use the doubly truncated length. 

{\it Type 2: McShane type gaps}

A slight word of warning: although these are similar to the gaps described in \cite{Mirzakhani} and \cite{Tan-Wong-Zhang}, they are fundamentally different in several ways. First of all, they are associated to an orthogeodesic, and not an embedded pair of pants. (Of course an embedded pair of pants is essentially equivalent to a simple {\it unoriented} orthogeodesic.) Here the gaps involve {\it any} oriented orthogeodesic, not necessarily simple and they are related to the half-pants alluded to earlier. We also point out that these only concern orthogeodesics that leave and return to a same boundary element. 

Although most of the discussion is mostly topological, it is sometime convenient to rely on the underlying geometry. In what follows, $\beta$ is implicitly a simple closed geodesic. The case where $\beta$ is a cusp is of course topologically identical, and can be done geometrically in a near identical fashion by replacing $\beta$ by a horocyclic neighborhood around $\beta$. 

As mentioned before, associated to such an orthogeodesic $\eta$, are two curves, the collection of which is denoted $\partial \eta$. These both are homotopic to the concatenation of $\eta$ with an arc of $\beta$ between the endpoints of $\eta$ (see Figure \ref{fig:halfpants1}).

\begin{figure}[h]
\leavevmode \SetLabels
\L(.495*.62) $\eta$\\%
\L(.49*.98) $\beta$\\%
\endSetLabels
\begin{center}
\AffixLabels{\centerline{\includegraphics[width=5cm]{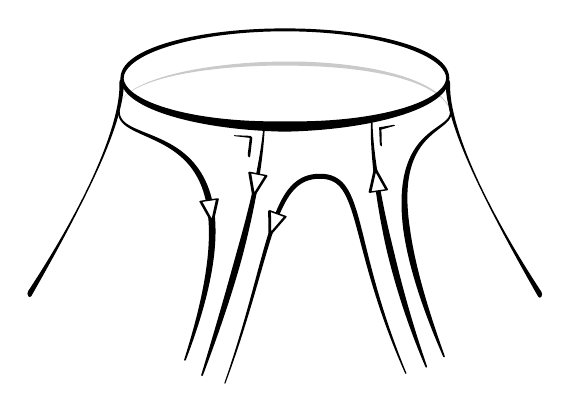}}}
\vspace{-24pt}
\end{center}
\caption{An orthogeodesic $\eta$ and the curves $\partial \eta$ near $\beta$}
\label{fig:halfpants1}
\end{figure}

Take one of the geodesic realizations of the curve, say $\alpha$, and observe that it bounds an immersed geodesic half-pants $H$, bounded on one end by $\alpha$, and on the other by $\eta$ and the segment $b$ of $\beta$ such that $b*\eta$ is freely homotopic to $\alpha$. The half-pants is embedded if and only if $\eta$ is simple.

Another way of seeing this is purely topological: if $\eta$ (as an arc) is peripheral to $\alpha$, then it is homotopic to a concatenation $c*\alpha*c^{-1}$ where $c$ is a path from $\beta$ to $\alpha$ (see Figure \ref{fig:peripheral}). 

\begin{figure}[H]
\leavevmode \SetLabels
\L(.303*.3) $\alpha$\\%
\L(.56*.57) $\eta$\\%
\L(.367*.62) $c$\\%
\endSetLabels
\begin{center}
\AffixLabels{\centerline{\includegraphics[width=7cm]{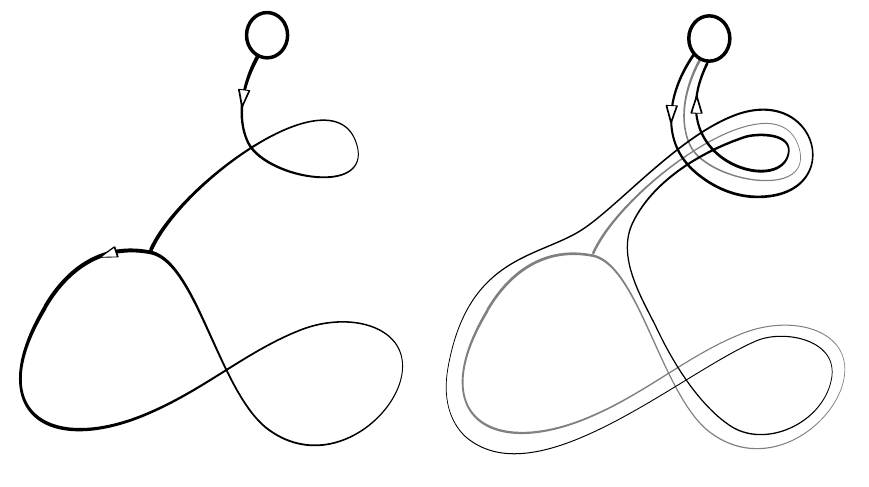}}}
\vspace{-24pt}
\end{center}
\caption{An arc $\eta$ peripheral to $\alpha$}
\label{fig:peripheral}
\end{figure}

The pre-immersed version of the cylinder is obtained by lifting $c$ and $\alpha$ to simple curves $\tilde{c}$ and $\tilde{\alpha}$. From this we obtain a lift of $\eta$ to $\tilde{\eta}$ (see Figure \ref{fig:peripherallift}).

\begin{figure}[H]
\leavevmode \SetLabels
\L(.273*.2) $\tilde{\alpha}$\\%
\L(.28*.4) $\tilde{\eta}$\\%
\L(.36*.44) $\tilde{c}$\\%
\L(.542*.59) $\eta$\\%
\L(.535*.4) $\alpha$\\%
\L(.607*.58) $c$\\%
\endSetLabels
\begin{center}
\AffixLabels{\centerline{\includegraphics[width=8cm]{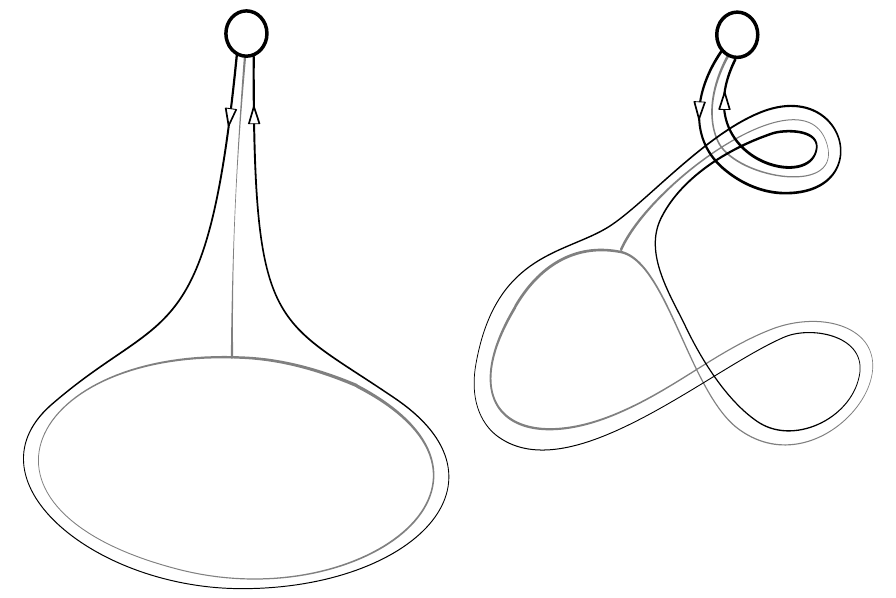}}}
\vspace{-24pt}
\end{center}
\caption{Lifting $\eta$, $c$ and $\alpha$ to $\tilde{H}$ (on the left)}
\label{fig:peripherallift}
\end{figure}
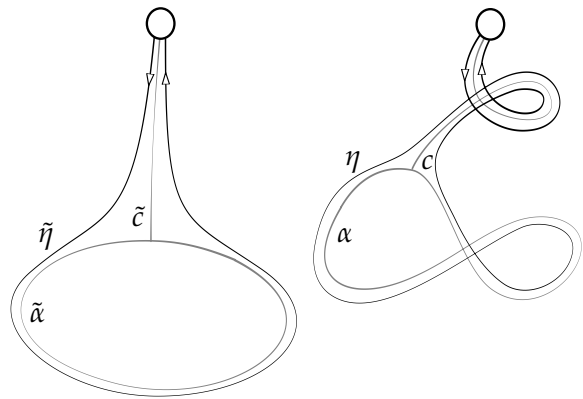

The half-pants $H$ are exactly the image via the immersion of the half-pants $\tilde{H}$ with boundary curves $\tilde{\alpha}$ on one side, and a piecewise geodesic boundary made of $\tilde{\eta}$ and a segment of $\tilde{\beta}$, the lift of $\beta$ (see Figure \ref{fig:halfpants2}). 

\begin{figure}[H]
\leavevmode \SetLabels
\L(.62*.19) $\tilde{\alpha}$\\%
\L(.48*.72) $\tilde{\eta}$\\%
\L(.36*.62) $\tilde{\beta}$\\%
\endSetLabels
\begin{center}
\AffixLabels{\centerline{\includegraphics[width=4cm]{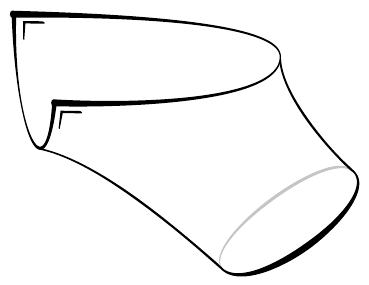}}}
\vspace{-24pt}
\end{center}
\caption{A lift of the geodesic half-pants}
\label{fig:halfpants2}
\end{figure}

The lifts of the immersed half-pants preserve the lengths of curves, and so it suffices to compute in $\tilde{H}$. For this we consider the simple geodesic orthoray $r$ leaving from $\tilde{\beta}$ and wrapping infinitely many times around $\tilde{\alpha}$ (in the same direction as $\tilde{\alpha}$). For those familiar with the gaps in McShane identities, this allows us to picture a gap  in between $\tilde{\eta}$ and $r$ (see Figure \ref{fig:halfgap}).

\begin{figure}[H]
\leavevmode \SetLabels
\L(.61*.19) $\tilde{\eta}$\\%
\L(.48*.82) $\tilde{\beta}$\\%
\L(.575*.62) $\tilde{r}$\\%
\endSetLabels
\begin{center}
\AffixLabels{\centerline{\includegraphics[width=6cm]{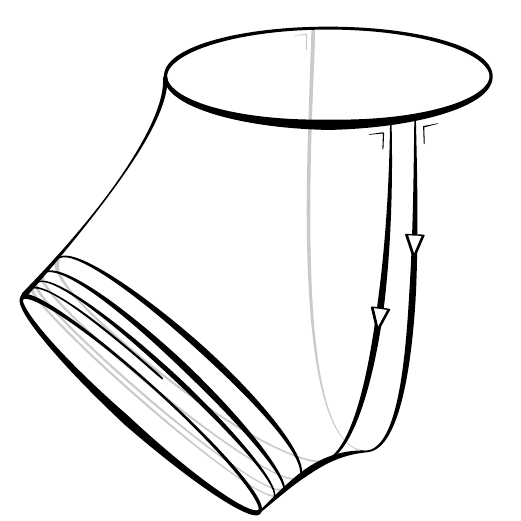}}}
\vspace{-24pt}
\end{center}
\caption{The distance between the basepoints of $\tilde{r}$ and $\tilde{\eta}$ is the sum of $\psi(\eta)$ and $\phi(\eta)/2$}
\label{fig:halfgap}
\end{figure}

The gap we consider here is {\it not} this whole gap however. From this gap, we remove the part of the gap coming from the Basmajian gap associated to $\eta$ (or $\tilde{\eta}$) to obtain a gap $\psi(\eta)$. This is best illustrated in the universal cover (see Figure \ref{fig:universalgap} where order not to introduce too much notation, the lifts of $\eta$, $\alpha$, $\beta$ and $r$ are also denoted $\tilde{\eta}$, $\tilde{\alpha}$, $\tilde{\beta}$ and $\tilde{r}$). The gap $\psi(\eta)$ associated to $\eta$ is the length of the segment as indicated.

\begin{figure}[h]
\leavevmode \SetLabels
\L(.43*.6) $\tilde{\eta}$\\%
\L(.385*.12) $\tilde{\alpha}$\\%
\L(.308*.945) $\tilde{\beta}$\\%
\L(.37*.24) $\tilde{r}$\\%
\L(.26*.34) $\psi(\eta)$\\%
\endSetLabels
\begin{center}
\AffixLabels{\centerline{\includegraphics[width=8cm]{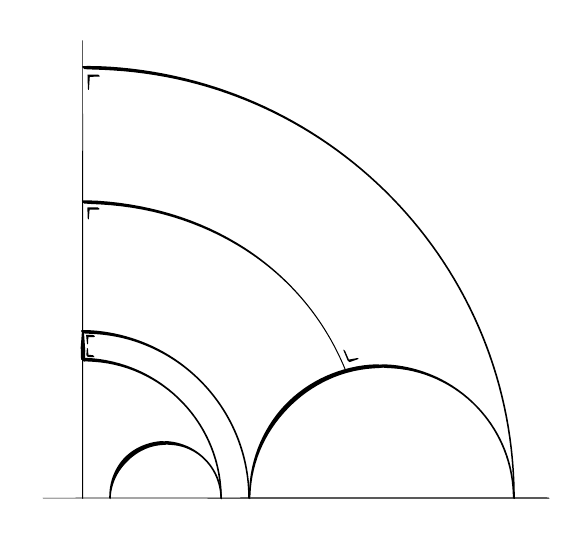}}}
\vspace{-24pt}
\end{center}
\caption{The gap $\psi(\eta)$ viewed in $\Hyp$}
\label{fig:universalgap}
\end{figure}

There is a dynamic interpretation of the gap: the gap is the set of (basepoints of) orthorays that lie entirely in $H$ until they make a loop around $\alpha$. The associated base vectors in $T^1(X)$ all belong to $L(\alpha)$. They might belong to other $L(\alpha')$ for $\alpha'\in M$, but they will form a loop around $\alpha$ before forming a loop around any other curve in $M$. This will be discussed in detail in the proofs of the identities.

The gap is bounded by two extremal cases: an orthoray that stays inside $H$ and wraps infinitely many times around $\alpha$, and an orthoray that corresponds to, on the completion of the surface, an ideal geodesic loop that is homotopic to $\alpha$. Note that, if $H$ is immersed and not embedded, there might be many (possibly infinitely many) such orthorays. The extremal ones we are interested in are the projections of the simple ones with this behavior in $\tilde{H}$. 

We can now compute $\psi(\eta)$. Up until now, the discussion has been essentially topological, and there are really two geometric situations to consider: when $\beta$ is a closed geodesic or when $\beta$ is a cusp. We begin with the first case, and the do the computation in the second case via a limiting argument. 

{\it Computing the gap when $\beta$ is a closed geodesic}

By cutting and unfolding the half-pants, we get a symmetric right-angled hexagon as in Figure \ref{fig:unfold}. By putting it in $\Hyp$ and unwrapping $r$, we can compute $\psi(\eta)$. 

\begin{figure}[H]
\leavevmode \SetLabels
\L(.63*.14) $\tilde{\alpha}$\\%
\L(.63*.805) $\tilde{\eta}$\\%
\L(.457*.3) $h$\\%
\endSetLabels
\begin{center}
\AffixLabels{\centerline{\includegraphics[width=12cm]{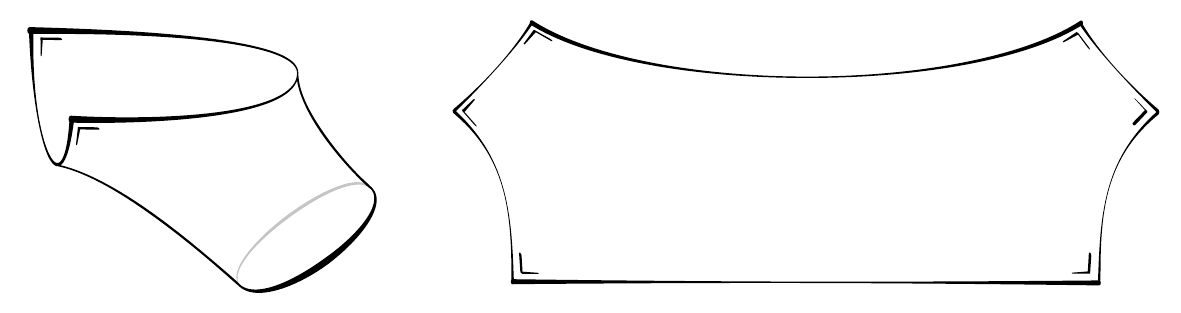}}}
\vspace{-24pt}
\end{center}
\caption{Cutting and unfolding the (lifted) half-pants}
\label{fig:unfold}
\end{figure}

In the following formulas, we use $\alpha$ for the length of $\alpha$ (which is the same as the length of $\tilde{\eta}$, and similarly for $\eta$. We also abbreviate $\psi(\eta)$ to $\psi$ and the Basmajian measure $\phi(\eta)$ to $\phi$. We denote by $x$ the quantity $\psi + \phi/2$ which is the length of the segment in between the base point of $\eta$ and the base point of $r$. And we denote by $y$ the complementary region to $x$ on the side of the hexagon on which $x$ lies. 
The quantity $h$ is one of the sides of the hexagon as indicated on Figures \ref{fig:unfold} and \ref{fig:pent}. 

$$
\sinh(h) \sinh(\psi) = 1.
$$

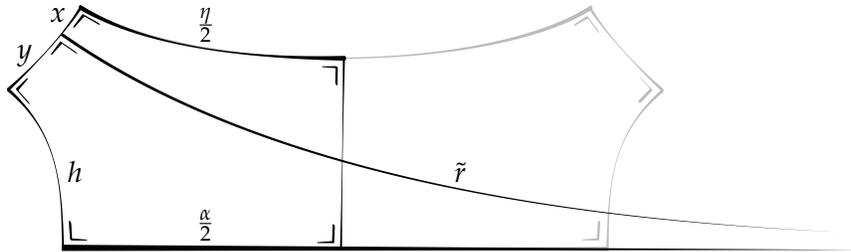
\begin{figure}[h]
\leavevmode \SetLabels
\L(.175*.883) $x$\\%
\L(.147*.75) $y$\\%
\L(.3*.845) $\frac{\eta}{2}$\\%
\L(.3*.12) $\frac{\alpha}{2}$\\%
\L(.19*.3) $h$\\%
\L(.52*.3) $\tilde{r}$\\%
\endSetLabels
\begin{center}
\AffixLabels{\centerline{\includegraphics[width=12cm]{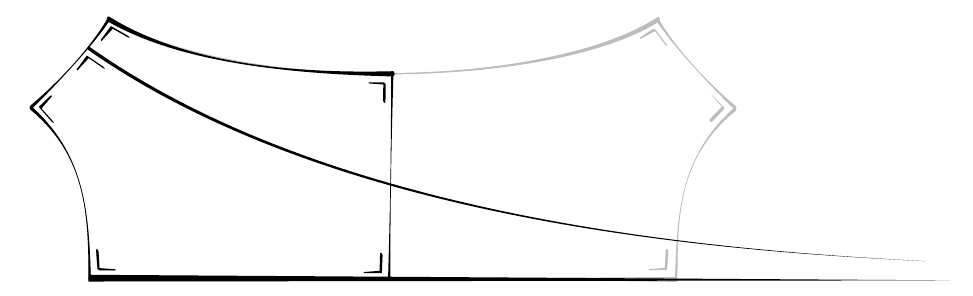}}}
\vspace{-24pt}
\end{center}
\caption{The pre-immersed pentagon}
\label{fig:pent}
\end{figure}

Now by looking in the right-angled pentagon that forms the left side of the hexagon (see Figure \ref{fig:pent}) and using the pentagon relation, we get
$$
\sinh(x + y) \sinh\left( \eta / 2 \right) = \cosh(\alpha / 2). 
$$
Putting this all together by elementary computations:
$$
x = \arcsinh\left(\frac{\cosh(\alpha/2)}{\sinh(\eta/2)}\right)- y = \arcsinh\left(\frac{\cosh(\alpha/2)}{\sinh(\eta/2)}\right) - \arcsinh\left(\frac{1}{\sinh(h)} \right).
$$
Now using the formula pentagon formula, $\sinh(h) = \cosh(\eta/2) / \sinh(\alpha/2)$ and so 
\begin{eqnarray*}
x &= &\arcsinh\left(\frac{\cosh(\alpha/2)}{\sinh(\eta/2)}\right) - \arcsinh\left(\frac{\sinh(\alpha/2)}{\cosh(\eta/2)} \right)\\
&=& \log\left( \coth\left(\frac{\eta}{2} \right)\right) + \log\left(\frac{\cosh(\alpha/2) + \sqrt{\cosh^2(\alpha/2) +\cosh^2(\eta/2) -1 }}{ \sinh(\alpha/2)+ \sqrt{\cosh^2(\alpha/2) +\cosh^2(\eta/2) -1 }} \right)\\
&=&\frac{ \phi}{2}+ \log\left(\frac{\cosh(\alpha/2) + \sqrt{\cosh^2(\alpha/2) +\cosh^2(\eta/2) -1 }}{ \sinh(\alpha/2)+ \sqrt{\cosh^2(\alpha/2) +\cosh^2(\eta/2) -1 }} \right)
\end{eqnarray*}
and because $\psi = x - \phi/2$, we get 
$$
\psi(\eta) = \log\left(\frac{\cosh(\alpha/2) + \sqrt{\cosh^2(\alpha/2) +\cosh^2(\eta/2) -1 }}{ \sinh(\alpha/2)+ \sqrt{\cosh^2(\alpha/2) +\cosh^2(\eta/2) -1 }} \right).
$$

{\it Half-trace notation}

It is sometimes convenient to use {\it traces} or {\it half-traces} of lengths. These are the traces and half-traces (or absolute value of) the corresponding matrices in $\PSL_2(\R)$. Note that to the gaps themselves, there is no corresponding surface group element, but there is a trace corresponding to the corresponding translation element in $\PSL_2(\R)$.

We denote the half-traces by
$$
a:= \cosh(\alpha/2), \, b:=\cosh(\beta/2)\mbox{ and }t:= \cosh(\eta/2)
$$
and then we can express $\psi(\eta)$ and $\phi(\eta)$ as follows:
\begin{eqnarray}
\phi(\eta) &=& \log\left(  \frac{t^2}{t^2-1} \right) \label{eq:halftrace1}\\
\psi(\eta) &=&  \log\left(  \frac{a + \sqrt{a^2-1 + t^2}}{\sqrt{a^2-1} + \sqrt{a^2-1 + t^2}} \right)\label{eq:halftrace2}\\
\beta &=& \log(b+ \sqrt{b^2-1}) \label{eq:halftrace3}
\end{eqnarray}

{\it Limiting to a cusp}

To compute the gap in the cusp case (this only concerns gaps of type 2), we use a standard limiting trick which involves multiplying the measure by a factor that varies in terms of the length of $\beta$. The resulting measure has the additional advantage of varying continuously from the boundary geodesic case to the cusp case. 

To do this, instead of viewing the measure on the boundary geodesic, we project it to the boundary of the collar of $\beta$, where the collar has been described previously in Remark \ref{rem:collar}. The collar boundary is not geodesic of course, but is a curve of constant curvature. By a standard computation, the boundary length of the collar of $\beta$ is $\beta \coth(\beta/2)$. The collar limits (continuously) to the standard horocyclic neighborhood of a cusp of boundary length $2$ as $\beta$ goes to $0$. 

A segment of length $\ell$ on $\beta$ projects to segment of length $\ell \coth(\beta/2)$ on the collar boundary. We get a new projected gaps given by $\coth(\beta/2) \psi(\eta)$ and $\coth(\beta/2) \phi(\eta)$, but the latter limits to $0$ as $\beta$ goes to $0$. The former however, even though $\psi(\eta)$ tends to $0$ as $\beta$ goes to $0$, doesn't disappear in the limit, and in fact, turns into a very nice expression. 

Before computing the limit, we need to know what we are computing the limit of. The length of the actual orthogeodesic $\eta$ goes to $\infty$ as $\beta$ goes to $0$, so we consider the twice truncated orthogeodesic $\tteta$: this is the orthogeodesic to and from the horocyclic neighborhood of the cusp where the initial and final geodesic segment which goes respectively from and to the cusp has been removed. This truncated orthogeodesic is the limit of the truncated orthogeodesic to and from the boundary of the collar of $\beta$ as the length of $\beta$ goes to $0$ (which we also denote $\tteta$). Note that, as described previously, the doubly truncated orthogeodesic may pass through the collar again, and in fact only does so if it winds around $\beta$, but we won't make use of that fact here. 

So we have 
$$
\eta = \tteta + 2 \arcsinh\left( \frac{1}{\sinh(\beta/2)} \right)
$$
and thus by a standard manipulation and limit argument:
$$
\lim_{\beta \to 0}\left( e^{\tteta/2} \frac{\cosh(\eta/2)}{\sinh(\beta/2)}\right)=1.
$$
With this in hand, and using the fact that 
$$
\lim_{\beta \to 0} \left(\coth(\beta/2) \psi(\eta) \right)= \lim_{\beta \to 0} \left(\coth(\beta/2) \sinh(\psi(\eta)) \right),
$$
one has
\begin{eqnarray*}
\lim_{\beta \to 0} \left(\coth(\beta/2) \psi(\eta)  \right)&=& \lim_{\beta \to 0}\left( \coth(\beta/2) \frac{e^{-\alpha/2}}{\cosh(\eta/2)} \right)\\
&=& \lim_{\beta \to 0} \left(\frac{e^{-\alpha/2}}{\sinh(\beta/2) \cosh(\eta/2)} \right) \\
&=& e^{-\alpha/2} \cdot e^{-\tteta/2}.
\end{eqnarray*}

For a surface with a cusp, to an orthogeodesic $\eta$ of truncated length $\tteta$ we thus associate the gap 
$$
\psi_\dt(\eta) = e^{-\frac{\alpha + \tteta}{2}}.
$$
\section{Proof of the identity}\label{sec:proof}
We now fix a marking $M$ of $\curves$, and take $X \in \M(\Sigma)$. The identity is a way of breaking up a boundary element of $X$ into gaps that depend on the behavior of the associated orthoray. 

We begin with the case when the boundary curve to be partitioned into gaps is a simple closed geodesic $\beta$. To each $p\in \beta$ corresponds an element $v_p$ of $T^1(X)$, the unique vector orthogonal to $\beta$ in $p$ and aiming inwards towards $X$. There is also the associated {\it orthoray} $\vec{r}_p$ given by forward exponentiating $v_p$. 

We follow $\vec{r}_p$ from its basepoint until one of three things happen: 
\begin{itemize}
\item It forms a loop around an element in $M$, in which case we stop and denote the associated geodesic arc $r_p$.
\item It hits $\partial X$, the boundary of $X$, and is freely homotopic to an orthogeodesic which we denote $\eta_p$. 
\item It wanders forever in $X$ without ever forming a loop around an element of $M$ or hitting $\partial X$. 
\end{itemize}

We will show in the sequel that the last case only happens for a measure $0$ set of points on $\beta$. 

{\it First case: a loop is made}

Let us begin with the first case, where we have an arc $r_p$, which ends exactly after having formed a loop around an element of $M$, say $\alpha_0$, for the first time. We can construct a homotopy class of arc and corresponding orthogeodesic by concatenating $r_p$ with the path that follows $r_p$ with the opposite orientation from the endpoint. We denote the corresponding arc and orthogeodesic by $\eta_0$.

\underline{Key observation:}  By construction: $\eta_0$ is peripheral to $\alpha_0$ and $p$ (or $\vec{r}_p$) belongs to the gap associated to $\eta_0$. 

The latter follows from the dynamics of the behavior of all orthorays in the gap associated to $\eta_0$. 

However, $\eta_0$ might not be unsupportive, that is $\eta_0$ might have a subloop belonging to $M$. If this is the case, we consider, following $\eta_0$ from its basepoint, the first time it forms a loop around an element of $M$, say $\alpha_1$, and denote by $r_1$ be the corresponding oriented subpath. We then construct a new homotopy class of orthogeodesic $\eta_1$ by following, from the endpoint of $r_1$, $\eta_0^{-1}$ back to $\beta$. Similarly to above, the key observation is that the entire gap of $\eta_0$ is contained in the gap of $\eta_1$. And by construction, $\eta_1$ is peripheral to $\alpha_1$.

We can now repeat the above process: if $\eta_1$ is {\it not} unsupportive, then it forms at least one loop around an element of $M$. We can then construct $\eta_2$, containing the gap associated to $\eta_1$, hence the gap associated to $\eta_0$, hence $p$, and so forth.

It is relatively straightforward to see that his process finishes after a finite number of steps, which only depend on the number of self-intersections of $\eta_0$ (but the only thing we really need is that it finishes). The point $p$ belongs to the gap of the resulting orthogeodesic, say $\eta_p$, which is both peripheral to an element of $M$, and unsupportive. 

{\it Second case: $\vec{r}_p$ returns to $\partial X$}

In this case, we stop $\vec{r}_p$ at $\partial X$ to obtain a finite geodesic arc, which in turn corresponds to an orthogeodesic $\eta_p$. We associate $p$ to the gap associated to $\eta_p$. Note that if it ends in a cusp, then the gap is of length $0$. 

To prove the identity we need to show that the gaps are disjoint, and that the set of for ever wandering orthorays is a measure $0$ subset of $\beta$. This latter fact - in the case where $\beta$ is a simple closed geodesic - follows from the same limit set argument as in \cite{Basmajian}, namely because the set of for ever wandering orthorays is a subset of the for ever wandering rays in the Basmajian identity. If however, $\beta$ is (the horocyclic boundary of) a cusp, this requires an additional argument.

\begin{lemma}
Let $X$ be a surface with a cusp, and let $\beta$ be a horocyclic neighborhood of a cusp. Suppose that $M\neq \emptyset$ and/or $\partial X$ contains a boundary geodesic. Then the set of points on $\beta$ which are the basepoints of orthorays of infinite length that are not supported on $M$ or that do not hit $\partial X$ in finite time are a subset of measure $0$ of $\beta$. In other words, the complementary regions of the gaps are of measure $0$. 
\end{lemma}

\begin{proof}
If $\partial X$ has boundary curves, then this follows from the same limit set argument as in \cite{Basmajian}. Suppose then that $\partial X$ consists only of cusps. In this case, the geodesic flow is ergodic and by hypothesis $M\neq \emptyset$. 

Consider $p$ with $\vec{r}_p$ which is infinitely long and never forms a loop around an element of $M$. By ergodicity, arbitrarily close to any orthoray $\vec{r}_p$ leaving from $p\in \beta$ is a geodesic that passes through $L(\alpha)$ for any $\alpha \in M$. This is because the latter forms a set of positive measure in $T^{1}(X)$. That means that arbitrarily close to the orthogonal vector $v_p$ corresponding to $\vec{r}_p$ a vector $v$ with angle arbitrarily close to $\frac{\pi}{2}$ that exponentiates to reach $L(\alpha)$. By continuity of geodesic behavior, arbitrarily close to $p$ is an orthoray leaving from $\beta$ belonging to a gap. Hence $p$ is isolated in $\beta$, proving the lemma.
\end{proof}

For the following lemma, we denote the gaps $\beta_i$, indexed by $i\in I$, and think of each gap as an open subset of $\beta$ (hence without its endpoints). 

\begin{lemma}
Let $\beta_i$ and $\beta_j$ be two different gaps. Then $\beta_i \cap \beta_j = \emptyset$. 
\end{lemma}

\begin{proof}
Suppose that $\beta_i$ and $\beta_j$ are of different type, say type 1 and type 2. Then if $p \in \beta_i$, the associated orthoray returns to $\beta$ without forming a loop in $M$, hence it cannot be of type 2. 

Note that two gaps of type 1 are disjoint, as there is a unique orthogeodesic associated to each homotopy class of arc between boundary elements. 

Now suppose that $\beta_i$ and $\beta_j$ are both of type 2.

It suffices to show that for all points of $\beta_i$, the first loop in $M$ that the associated orthoray forms is freely homotopic to $\alpha_i$, and that the associated arc constructed above is always homotopic to $\eta_i$. 

Consider the boundary point $p_i$ of $\beta_i$ corresponding to an ideal orthoray freely homotopic to $\alpha_i$. By continuity, for all points in $\beta_i$ sufficiently close to $p_i$, the associated orthorays first form a loop around $\alpha_i$ creating an associated arc homotopic to $\eta_i$. Note that throughout the gap, the orthoray will continue to form this loop. We just need to analyse what happens for in each self-intersection point of the orthoray that occurs before this loop is formed.

Now suppose $p'$ is the first point in the gap such that the associated orthoray $\vec{r}_{p'}$ forms another loop , say $\alpha'$, before looping around $\alpha_i$. By continuity, as this is the first point, it occurs exactly at a previous intersection point where the associated loop is homotopic to $\alpha_i$. Hence, it is easy to see that $\alpha'$ is a subloop of $\alpha_i$. So either it is homotopic to $\alpha_i$, in which we are happy, or its a proper subloop, in which case its homotopy class cannot belong to $M$ by definition of a coherent marking. Note that this is exactly the crucial moment where coherent markings come into play. And further note that this {\it never} happens if $\eta_i$ is simple.

\begin{figure}[H]
\leavevmode \SetLabels
\L(.147*.7) $\eta_i$\\%
\L(.37*.3) $\alpha'$\\%
\endSetLabels
\begin{center}
\AffixLabels{\centerline{\includegraphics[width=14cm]{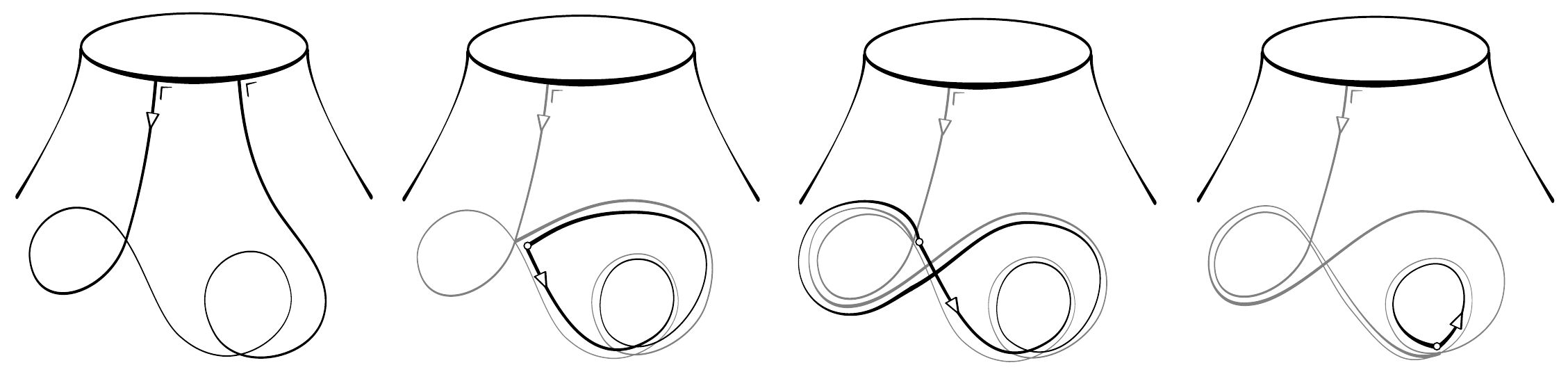}}}
\vspace{-24pt}
\end{center}
\caption{An example orthogeodesic $\eta_i$ and subloops that appear along orthorays in the gaps}
\label{fig:subloops}
\end{figure}

This was only the first critical point (meaning the first time an extra intersection point appeared). As we continue to push orthorays along the gap, the same process occurs again, and again, by the same argument, the new loop that appears is freely homotopic to $\alpha_i$ or to one of its subloops. This process is illustrated in Figure \ref{fig:subloops}.

Hence the only loops that appear that belong to $M$ are all freely homotopic to $\alpha_i$. Thus there is no first point of the gap with a different behavior, which proves that two gaps of type 2 cannot overlap. 
\end{proof}

This proves the (abstract) identities (Theorem \ref{thm:abstract}), and in the next section, we restate and discuss the quantified identities, obtained by replacing the measures of the gaps with their computed values. 

\section{The full identities and final remarks}\label{sec:exampleidentities}

Now that the index set has been properly introduced, that gaps have been quantified, and that the proof that they form a full measure partition of the boundary length is done, we can state the identities. 

\subsection{The general identities}

We begin with a full identity, which decomposes the actual boundary of $X$. In particular if $\partial X$ contains a cusp, it is considered a geodesic of length $0$, all orthogeodesics that have at least one endpoint in it are of infinite length, and contribute nothing to the sum. 

We recall the notation: for a given coherent marking $M$, $\OM$ denotes the orthogeodesics unsupportive of $M$ (for implicit $X$). And $\OMA$ is the subset of $\OM$ which is peripheral to $\alpha$. Note that by definition of being peripheral, orthogeodesics in $\OMA$ leave and return to the same connected component of $\partial X$.

\begin{theorem}
Let $X\in \M(\Sigma)$ be a hyperbolic surface of finite type with $\partial X \neq \emptyset$ containing at least one simple closed geodesic, and $M$ a coherent marking of $\curves$. The lengths of curves and orthogeodesics of $X$ satisfy
\begin{eqnarray*}
\ell(\partial X) = &
\Sum_{\eta \in \OO_{\hspace{-0.1em}M}} & \log \left( \coth^2 (\eta/2) \right)\\
 +& \Sum_{\alpha \in M} \,\, \Sum_{\eta \in \OO_{\hspace{-0.1em}M}^{\alpha}}  &\log\left( \frac{\cosh(\alpha/2) + \sqrt{\cosh^2(\alpha/2) + \cosh^2(\eta/2)-1}}{\sinh(\alpha/2) + \sqrt{\cosh^2(\alpha/2) + \cosh^2(\eta/2)-1}} \right).
\end{eqnarray*}
\end{theorem}

The condition on the boundary containing at least one simple closed geodesic is to avoid the above identity be a sum of $0$s equal to $0$.

The above identity is really a sum of identities taken over the connected components of $\partial X$ that are simple closed geodesics. If we fix a boundary element $\beta$, the identity associated to $\beta$ takes two forms, which only depends on whether $\beta$ is realized on $X$ or as a cusp or as a geodesic.

For $M$ is a coherent marking, $\OBM$ will denote the subset of $\OM$ consisting of orthogeodesics leaving from $\beta$. And $\OBA$ is the subset of $\OB_M$ which is peripheral to $\alpha$. As before, by definition, elements of $\OBA$ leave and return to $\beta$.

If $\beta$ is a boundary simple closed geodesic of $X$, it satisfies:
\begin{eqnarray*}
\ell(\beta) = &
\Sum_{\eta \in \OBM} & \log \left( \coth^2 (\eta/2) \right)\\
 +& \Sum_{\alpha \in M} \,\, \Sum_{\eta \in \OBA}  &\log\left( \frac{\cosh(\alpha/2) + \sqrt{\cosh^2(\alpha/2) + \cosh^2(\eta/2)-1}}{\sinh(\alpha/2) + \sqrt{\cosh^2(\alpha/2) + \cosh^2(\eta/2)-1}} \right).
\end{eqnarray*}

Note that using the half-trace convention and computations (Equations \ref{eq:halftrace1}, \ref{eq:halftrace2} and \ref{eq:halftrace2}), the above becomes Equation \ref{eq:halftraceformulation}. 

If $\beta$ is a cusp:
\begin{equation}\label{eq:cusp0}
2 = 2 \Sum_{\eta \in \OBM} e^{-\teta} +  \Sum_{\alpha \in M} \,\, \left(\Sum_{\eta \in \OBA} e^{-\frac{\alpha+\tteta}{2}}\right)
\end{equation}
where $\teta$ is the truncated length of $\eta$ and, for orthogeodesics that leave and return to $\beta$, $\tteta$ is the doubly truncated length. Note that in this case, every element of $\OBM$ appears exactly once. Indeed, an orthogeodesic in $\OBM$ returning to $\beta$ has infinite truncated length, and thus contributes $0$ to the lefthand sum, but does contribute to the right hand sum with its finite doubly truncated length. If however an orthogeodesic does not return to $\beta$, it cannot be peripheral to an element of $M$. Rewriting this sum taking this into account, and reorganizing the last terms, gives us:

\begin{equation}\label{eq:cusp1}
1 = \Sum_{\eta \in \OBM^{\mbox{\tiny{np}}}} e^{-\teta} +  \frac{1}{2}\Sum_{\alpha \in M} e^{-\frac{\alpha}{2}} \left( \Sum_{\eta \in \OBA} e^{-\frac{\tteta}{2}}\right)
\end{equation}
where $ \OBM^{\mbox{\tiny{np}}}$ is the set of non-peripheral elements of $\OBM$. In particular
$$
\OBM = \OBM^{\mbox{\tiny{np}}} \dot\cup\left(\dot\cup_{\alpha\in M} \OBA\right).
$$
With this viewpoint, Equation \ref{eq:cusp1} gives a nice probability measure on orthogeodesics in $\OBM$. The measure associated to an orthogeodesic $\eta$ in $\OBM^{\mbox{\tiny{np}}}$ is $e^{-{\teta}}$ and to an orthogeodesic $\eta$ in $\OBA$ is $e^{-\frac{\alpha+ \tteta}{2}}/2$. 

\begin{remark} There is something to check in as we have inverted (possibly) infinite sums. To show the sums are indeed the same, it suffices to show that the sum of $e^{-\tteta/2}$ for all $\eta$ peripheral to $\alpha$ and unsupportive of $M$ converges. Depending on $M$, the index set for this sum is a subset of the set $\OB_\alpha^\alpha$ of orthogeodesics peripheral to $\alpha$ and unsupportive of $\alpha$. Now, for {\it any} $\alpha \in \curves$, from \ref{eq:cusp0} with $M=\{\alpha\}$:
$$
2 = 2 \Sum_{\eta \in \OB_\alpha^{\mbox{\tiny{np}}}} e^{-\teta} + e^{-\frac{\alpha}{2}} \left( \Sum_{\eta \in \OB_\alpha^\alpha} e^{-\frac{\tteta}{2}}\right).
$$
In particular, the sum $\sum_{\eta \in \OB_\alpha^\alpha} e^{-\frac{\tteta}{2}}$ converges. 
\end{remark}

If all boundary elements of $X$ are cusps, then the first sum of Equation \ref{eq:cusp1} disappears. In this case we have 

\begin{equation}\label{eq:cusp2}
1 = \frac{1}{2}\Sum_{\alpha \in M} e^{-\frac{\alpha}{2}} \left( \Sum_{\eta \in \OBA} e^{-\frac{\tteta}{2}}\right)
\end{equation}

This time we can also view Equation \ref{eq:cusp2} as a type of probability measure on elements in $M$, where the measure on $\alpha$ is given by
$$
 \frac{1}{2} e^{-\frac{\alpha}{2}} \left( \Sum_{\eta \in \OBA} e^{-\frac{\tteta}{2}}\right). 
 $$
 
 If $X$ has $n$ cusps, by adding up the terms of Equation \ref{eq:cusp2}, one obtains 
 \begin{equation}\label{eq:cusp3}
n = \frac{1}{2}\Sum_{\alpha \in M} e^{-\frac{\alpha}{2}} \left( \Sum_{\eta \in \OMA} e^{-\frac{\tteta}{2}}\right)
\end{equation}

\subsection{Relationship to previous identities}\label{ss:relationships}

{\it The Basmajian identity}

When $M=\emptyset$, there are no marked curves, hence no peripheral orthogeodesics, and $\OM = \OO$. The identity becomes 
$$
\ell(\partial X) = 
\Sum_{\eta \in \OO} 2 \log \left( \coth (\eta/2) \right)
$$
which is exactly the identity proved by Basmajian \cite{Basmajian}. Note it only has content when $\partial X$ contains at least one simple closed geodesic, and has no content when $\partial X$ consists solely of cusps. Note there is a way of making sense of the identity when $X$ only has cusps, or even cone-points, obtained by regrouping terms: see the results from \cite{BPT}.

Note this is not the same time the Basmajian identity has been linked to another identity. Bridgeman and Tan \cite{Bridgeman-Tan} showed how the Basmajian identity is linked to moments of the Liouville measure. In fact it is the first term of a family of identities that you can derive from this point of view, where the second term is the Bridgeman identity \cite{Bridgeman} (and Bridgeman-Kahn \cite{Bridgeman-Kahn} in higher dimensions). Also see \cite{Calegari}. 

{\it McShane type identities}

When $M$ is the set of all simple closed curves (primitive, with both orientations, including those peripheral to boundary), $\OM$ becomes the set of simple orthogeodesics $\OO_0$. This is because any non-simple orthogeodesic supports a closed curve, and as such, supports a simple closed curve. The last implication follows from the observation that any non-simple closed curve has a simple (non-trivial) subloop, and hence it supports a simple closed curve.

In McShane type identities, the sums are over embedded pants with either one or two of its cuffs peripheral to boundary.

Consider an embedded pair of pants $P$ with a single peripheral cuff $\beta$ and two other cuffs $\alpha_1$ and $\alpha_2$. $P$ There are two simple orthogeodesics $\eta$ and $\eta^{-1}$ leaving from $\beta$ and contained in $P$: they have the same traces but opposite orientations (see Figure \ref{fig:mcshanepants}). The measure associated to $P$ is the sum of the measures associated to $\eta$ and $\eta^{-1}$. If $\beta$ is a geodesic, this constitutes 6 terms: the measures $\phi(\eta)$ and $\phi(\eta^{-1})$ (the Basmajian type terms), the terms associated to $\eta$ and $\alpha_1$, resp. $\alpha_2$, and the terms associated to $\eta^{-1}$ and $\alpha_1$, resp. $\alpha_2$. If $\beta$ is a cusp, then the Basmajian type terms disappear.

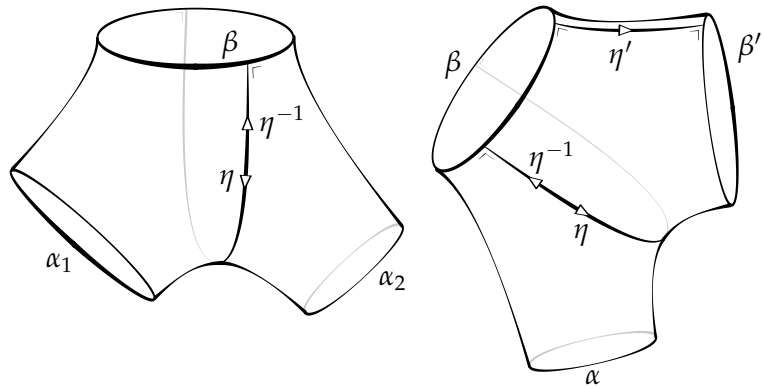
\begin{figure}[H]
\leavevmode \SetLabels
\L(.402*.65) $\eta^{-1}$\\%
\L(.368*.515) $\eta$\\%
\L(.22*.305) $\alpha_1$\\%
\L(.504*.26) $\alpha_2$\\%
\L(.68*.01) $\alpha$\\%
\L(.37*.85) $\beta$\\%
\L(.56*.8) $\beta$\\%
\L(.81*.84) $\beta'$\\%
\L(.67*.38) $\eta$\\%
\L(.632*.56) $\eta^{-1}$\\%
\L(.70*.82) $\eta'$\\%
\endSetLabels
\begin{center}
\AffixLabels{\centerline{\includegraphics[width=10cm]{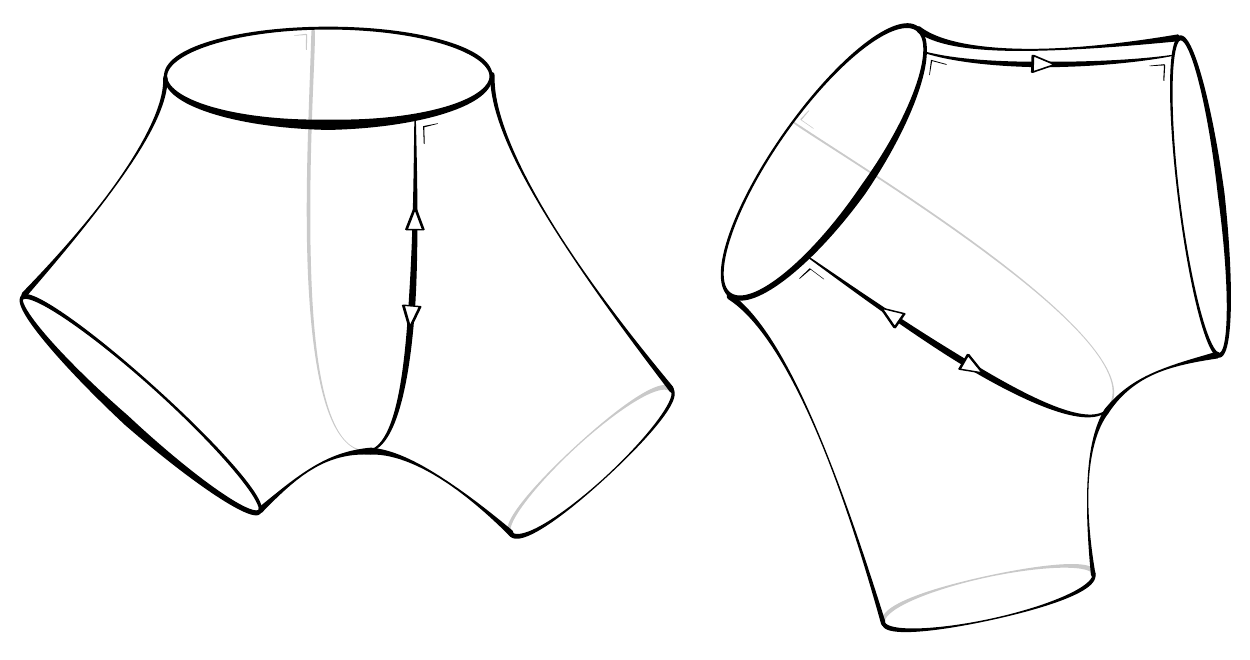}}}
\vspace{-24pt}
\end{center}
\caption{The pairs of pants $P$ on the left and $Q$ on the right.}
\label{fig:mcshanepants}
\end{figure}

To an embedded pair of pants $Q$ with two peripheral cuffs, $\beta$ and $\beta'$, and an additional cuff $\alpha$, in the McShane setup there is a gap associated to $\beta$ and one associated to $\beta'$. There are three simple orthogeodesics in $Q$ that leave from $\beta_1$ (see Figure \ref{fig:mcshanepants}). Two of them, say $\eta$ and $\eta^{-1}$ return to $\beta$, and like before have the same traces and different orientations. The 6 gaps associated to them are as before, where $\alpha$ plays the part of $\alpha_1$ and $\beta'$ plays the part of $\alpha_2$. There is another simple orthogeodesic, say $\eta'$, that goes from $\beta$ to $\beta'$, and to it we associate an additional Basmajian type gap. Again, the sum of the usual McShane measure is the sum of these 7 gaps this time. As before, if $\beta$ is a cusp, the Basmajian type terms disappear.

{\it Final remarks}

The Basmajian and McShane identities have been studied in many contexts. For instance, the Basmajian identity has been generalized to inequalities on maximal representations \cite{Fanoni-Pozzetti}, to higher Teichm\"uller theory \cite{Vlamis-Yarmola}, and to situations with limit sets of different Hausdorff dimension \cite{He}. The McShane identity found its first generalizations in the work of Bowditch \cite{Bowditch2}, Mirzakhani \cite{Mirzakhani} and Tan-Wong-Zhang \cite{Tan-Wong-Zhang}, the latter having generalized it to cone surfaces. Bowditch also provided another proof of the original torus identity \cite{Bowditch1}, recently generalized by Labourie-Tan \cite{Labourie-Tan}. Further generalizations include \cite{Labourie-McShane}, \cite{Huang-Sun}, and \cite{Charette-Goldman}. A very useful survey is \cite{Bridgeman-Tan2}, although of course it doesn't take into account the most recent work. 

It might be interesting to see to what extent the identities investigated in this paper generalize to other contexts. As mentioned previously, there is direct link between the Bridgeman identity and the Basmajian identity \cite{Bridgeman-Tan}, and the two identities share the same index set. In the current identities, the measures are always sums of Basmajian terms (the McShane type terms being infinite sums of Basmajian type terms that share the same initial dynamic behavior). Hence there is an immediate abstract identity related to the Bridgeman identity, obtained by summing the terms corresponding to the same index set. In a related direction, the Luo-Tan identity \cite{Luo-Tan} decomposes the volume of the unit tangent bundle (also see \cite{Hu-Tan}). The index set in these identities is the set of embedded pairs of pants. It seems reasonable to hope to be able to generalize the techniques of the current paper to this context to obtain identities with certain types of immersed pairs of pants.

{\em Address:}\\
Department of Mathematics, University of Luxembourg, Luxembourg \\
{\em Email:}
 \href{mailto:hugo.parlier@unifr.ch}{hugo.parlier@uni.lu}\\


\begin{thebibliography}{99}

\bibitem{Basmajian} Basmajian, Ara. The orthogonal spectrum of a hyperbolic manifold. Amer. J. Math. 115 (1993), no. 5, 1139--1159.

\bibitem{Basmajian2} Basmajian, Ara. The stable neighborhood theorem and lengths of closed geodesics. Proc. Amer. Math. Soc. 119 (1993), no. 1, 217--224.

\bibitem{BPT} Basmajian, Ara, Parlier, Hugo and Tan, Ser Peow. Prime orthogeodesics, concave cores and families of identities on hyperbolic surfaces. Preprint, 2020. 

\bibitem{Bowditch1} Bowditch, Brian H. A proof of McShane's identity via Markoff triples. Bull. London Math. Soc. 28 (1996), no. 1, 73--78.

\bibitem{Bowditch2} Bowditch, Brian H. A variation of McShane's identity for once-punctured torus bundles. Topology 36 (1997), no. 2, 325--334.

\bibitem{Bridgeman} Bridgeman, Martin. Orthospectra of geodesic laminations and dilogarithm identities on moduli space. Geom. Topol. 15 (2011), no. 2, 707--733.

\bibitem{Bridgeman-Kahn} Bridgeman, Martin and Kahn, Jeremy. Hyperbolic volume of manifolds with geodesic boundary and orthospectra. Geom. Funct. Anal. 20 (2010), no. 5, 1210--1230.

\bibitem{Bridgeman-Tan} Bridgeman, Martin and Tan, Ser Peow. Moments of the boundary hitting function for the geodesic flow on a hyperbolic manifold. Geom. Topol. 18 (2014), no. 1, 491--520.

\bibitem{Bridgeman-Tan2} Bridgeman, Martin and Tan, Ser Peow. Identities on hyperbolic manifolds. Handbook of Teichm\"uller theory. Vol. V, 19--53, IRMA Lect. Math. Theor. Phys., 26, Eur. Math. Soc., Z\"urich, 2016. 

\bibitem{BuserBook} Buser, Peter. Geometry and spectra of compact Riemann surfaces. Progress in Mathematics, 106. Birkh\"auser Boston, Inc., Boston, MA, 1992.

\bibitem{Calegari} Calegari, Danny. Chimneys, leopard spots and the identities of Basmajian and Bridgeman. Algebr. Geom. Topol. 10 (2010), no. 3, 1857--1863.

\bibitem{Charette-Goldman} Charette, Virginie and Goldman, William M. McShane-type identities for affine deformations. Ann. Inst. Fourier (Grenoble) 67 (2017), no. 5, 2029--2041.

\bibitem{Fanoni-Pozzetti} Fanoni, Federica and Pozzetti, Maria Beatrice. Basmajian-type inequalities for maximal representations. J. Differential Geom., to appear.

\bibitem{He} He, Yan Mary. Basmajian-type identities and Hausdorff dimension of limit sets. Ergodic Theory Dynam. Systems 38 (2018), no. 6, 2224--2244.

\bibitem{Hu-Tan} Hu, Hengnan and Tan, Ser Peow. New identities for small hyperbolic surfaces. Bulletin of the London Mathematical Society 46.5 (2014): 1021--1031.

\bibitem{Huang-Sun} Huang, Yi and Sun, Zhe. McShane identities for Higher Teichm\"uller theory and the Goncharov-Shen potential. Preprint, 2019. 

\bibitem{Labourie-McShane} Labourie, Fran\c{c}ois and McShane, Greg. Cross ratios and identities for higher Teichm\"uller-Thurston theory. Duke Math. J. 149 (2009), no. 2, 279--345.

\bibitem{Labourie-Tan} Labourie, Fran\c{c}ois and Tan, Ser Peow. The probabilistic nature of McShane's identity: planar tree coding of simple loops. Geom. Dedicata 192 (2018), 245--266.

\bibitem{Luo-Tan} Luo, Feng and Tan, Ser Peow. A dilogarithm identity on moduli spaces of curves. J. Differential Geom. 97 (2014), no. 2, 255--274. 

\bibitem{McShane} McShane, Greg. Simple geodesics and a series constant over Teichm\"uller space. Invent. Math. 132 (1998), no. 3, 607--632. 

\bibitem{Mirzakhani} Mirzakhani, Maryam. Simple geodesics and Weil-Petersson volumes of moduli spaces of bordered Riemann surfaces. Invent. Math. 167 (2007), no. 1, 179--222.

\bibitem{Tan-Wong-Zhang} Tan, Ser Peow, Wong, Yan Loi and Zhang, Ying. Generalizations of McShane's identity to hyperbolic cone-surfaces. J. Differential Geom. 72 (2006), no. 1, 73--112.

\bibitem{Vlamis-Yarmola} Vlamis, Nicholas G. and Yarmola, Andrew. Basmajian's identity in higher Teichm\"uller-Thurston theory. J. Topol. 10 (2017), no. 3, 744--764. 

\end{thebibliography}
\end{document}